\documentclass[preprint, 11pt]{elsarticle}
\usepackage{amsmath}
\usepackage{amssymb}
\usepackage{amscd}
\usepackage{amsthm}
\usepackage{todonotes}
\usepackage{verbatim}
\usepackage[all]{xy} 

\def\m{{\mathfrak m}} 

\def\p{{\mathfrak p}}    

\def\v{{\mathfrak v}}

\DeclareMathOperator{\Spec}{Spec}

\begin{document}

\newtheorem{theorem}{Theorem}[section]
\newtheorem{lemma}[theorem]{Lemma}
\newtheorem{proposition}[theorem]{Proposition}
\newtheorem{corollary}[theorem]{Corollary}
\newtheorem{problem}[theorem]{Problem}

\theoremstyle{definition}
\newtheorem{example}[theorem]{Example}
\newtheorem{defi}[theorem]{Definitions}
\newtheorem{definition}[theorem]{Definition}
\newtheorem{remark}[theorem]{Remark}
\newtheorem{discussion}[theorem]{Discussion}
\newtheorem{fact}[theorem]{Fact}
\newtheorem{ex}[theorem]{Example}
\newtheorem{question}[theorem]{Question}
\newtheorem{notation}[theorem]{Notation}
\newtheorem{setting}[theorem]{Setting}

\renewcommand{\thedefi}{}

\long\def\alert#1{\smallskip{\hskip\parindent\vrule%
\vbox{\advance\hsize-2\parindent\hrule\smallskip\parindent.4\parindent%
\narrower\noindent#1\smallskip\hrule}\vrule\hfill}\smallskip}

\def\ff{\mathfrak}
\def\Proj{\mbox{\rm Proj}}
\def\type{\mbox{ type}}
\def\Hom{\mbox{ Hom}}
\def\rank{\mbox{ rank}}
\def\Ext{\mbox{ Ext}}
\def\Ker{\mbox{ Ker}}
\def\Max{\mbox{\rm Max}}
\def\Min{\mbox{\rm Min}}
\def\End{\mbox{\rm End}}
\def\xpd{\mbox{\rm xpd}}
\def\Ass{\mbox{\rm Ass}}
\def\emdim{\mbox{\rm emdim}}
\def\epd{\mbox{\rm epd}}
\def\repd{\mbox{\rm rpd}}
\def\ord{{\rm ord}}
\def\hgt{{\rm ht}}

\def\edeg{{\rm e}}

\begin{frontmatter}

\title{Asymptotic properties of 
infinite directed unions \\ of local quadratic transforms}
\author[Bill]{William Heinzer}
\ead{heinzer@purdue.edu}
\address[Bill]{Department of Mathematics, Purdue University, West Lafayette, Indiana 47907}
\author[Bruce]{Bruce Olberding}
\ead{olberdin@nmsu.edu}
\address[Bruce]{Department of Mathematical Sciences, New Mexico State University, Las Cruces, NM 88003-8001}
\author[Matt]{Matthew Toeniskoetter}
\ead{mtoenisk@purdue.edu}
\address[Matt]{Department of Mathematics, Purdue University, West Lafayette, Indiana 47907}

\begin{abstract}
Let $(R, \m)$ be a  regular local ring of dimension at least 2.  
For each valuation domain birationally dominating $R$,
there is an associated sequence $\{R_n\}$ of local quadratic transforms of $R$.
We consider the case where this sequence $\{ R_n \}_{n \ge 0}$ is infinite 
and examine properties of the integrally closed local domain $S = \bigcup_{n \ge 0}R_n$ in the case where $S$ is not a valuation domain.
For this sequence, there is an associated boundary valuation ring $V = \bigcup_{n \ge 0} \bigcap_{i \ge n} V_i$, where $V_i$ is the order valuation ring of $R_i$.
There exists a unique minimal proper Noetherian overring $T$ of $S$.
$T$ is the regular Noetherian UFD obtained by localizing outside the maximal ideal of $S$ and $S = V \cap T$.
In the present paper, we define functions $w$ and $\edeg$, where $w$ is the asymptotic limit of the order valuations and $\edeg$ is the limit of the orders of transforms of principal ideals.
We describe $V$ explicitly in terms of $w$ and $\edeg$ and prove that $V$ is either rank $1$ or rank $2$.
We define an invariant $\tau$ associated to $S$ that is either a positive real number or $+\infty$.
If $\tau$ is finite, then $S$ is archimedean and $T$ is not local.
In this case, the function $w$ defines the rank $1$ valuation overring $W$ of $V$ and $W$ dominates $S$.
The rational dependence of $\tau$ over $w (T^{\times})$ determines whether $S$ is completely integrally closed and whether $V$ has rank $1$.
We give examples where $S$ is completely integrally closed.
If $\tau$ is infinite, then $S$ is non-archimedean and $T$ is local.
In this case, the function $e$ defines the rank $1$ valuation overring $E$ of $V$.
The valuation ring $E$ is a DVR and $E$ dominates $T$, and in certain cases we prove that $E$ is the order valuation ring of $T$.

\end{abstract}


\end{frontmatter}


\section{Introduction  and summary}  \label{sec1}
 
Let $(R,\m)$ be a regular local ring and let $ S = \bigcup_{n \ge 0}R_n$  be an infinite 
directed union of local quadratic transforms  as in the abstract.  In \cite{HLOST},  the authors  consider
 ideal-theoretic properties of the integral domain  $S$.
The ring $S$ is local and integrally closed. 
Abhyankar proves in \cite[Lemma 12, p. 337]{Abh} that if $\dim R = 2$, then $S$ is a valuation domain.
However, if $\dim S \ge 3$, then $S$ is no longer a valuation domain in general.
In the case where $\dim R \ge 3$, David Shannon examines properties of $S$ in \cite{Sha}.
This motivates the following definition.

\begin{definition}
Let $R$ be a regular local ring with $\dim R \ge 2$ and let $\{R_n\}_{n \ge 0}$ be an infinite sequence of regular local rings, where $R = R_0$ and $R_{n+1}$ is a local quadratic transform of $R_{n}$ for each $n \ge 0$.
Then  $S = \bigcup_{n \ge 0}R_n$  is a {\it Shannon extension} of $R$. 
\end{definition}

Let $S$ be a Shannon extension of $R = R_0$ and let $F$ denote the field of fractions of $R$.  
Associated to 
each of the regular local rings $(R_i,\m_i)$,  there is a rank 1 discrete valuation ring $V_i$ 
defined by  the {\it order function}  $\ord_{R_i}$,  where for $x \in R_i$, 
$\ord_{R_i}(x) = \sup \{ n \mid x \in \m_i^n \}$.
The family $\{V_i\}_{i=0}^\infty$ determines 
a unique set  
\begin{equation*}\label{equation V}
	V ~   =   ~    \bigcup_{n \ge 0}~ \bigcap_{i \ge n} V_i = \{ a \in F ~ | ~  \ord_{R_i }(a) \ge 0 \text{ for } i \gg 0 \}.\end{equation*}

The set $V$  consists of the elements in $F$ that are in all but finitely many of the  $V_i$.
In \cite{HLOST},  the authors prove that $V$ is a valuation domain that birationally
dominates $S$,  and call $V$ the {\it boundary valuation ring} of the Shannon extension $S$.

In Section~\ref{sec2}, we review the concept of the transform of an ideal, and in Sections~\ref{sec3} and \ref{essential}, we discuss previous results on Shannon extensions.
Theorem~\ref{hull} describes an intersection decomposition $S = V \cap T$ of a Shannon extension $S$, where $V$ is the boundary valuation of $S$ and $T$ is the intersection of all the DVR overrings of $R$ that properly contain $S$.
In Setting~\ref{setting 1},  we fix  notation  to  use  
 throughout  the remainder of the paper.
In Discussion~\ref{3.2}, we describe conditions in order that $S$ be a valuation domain.
 
 In Section~\ref{sec4} we
 consider asymptotic behavior of the family $\{\ord_{R_n}\}_{n \ge 0}$   of order
 valuations of a Shannon extension  $S = \bigcup_{n \ge 0}R_n$.   
For nonzero $a \in S$, we fix some $n$ such that $a \in R_n$ and define
  $\displaystyle \edeg (a) = \lim_{i \rightarrow \infty} \ord_{R_{n +i}} ((a R_n)^{R_{n+i}})$,  where 
  $(a R_n)^{R_{n+i}}$ denotes the transform in $R_{n+i}$ of the ideal $aR_n$. 
For nonzero elements $a, b \in S$, we define   $\edeg (\frac{a}{b}) = \edeg (a) - \edeg (b)$.  
In Lemmas~\ref{e facts} and \ref{factorization},  we prove the function $\edeg$ is well 
defined and that $\edeg$  describes   factorization properties of elements in $S$.  

We fix an element $x \in S$ such that $xS$ is primary to the maximal ideal of $S$.
Theorem~\ref{limit exists} proves that the asymptotic limit
$$\lim_{n \rightarrow \infty}  ~~\frac{\ord_n (q)}{\ord_n (x)}$$
exists for every nonzero element  $q$ in the quotient field $F$ of $S$, but may take values $\pm \infty$.
We denote this function
$w : F \rightarrow {\mathbb{R}} ~  \cup   ~  \{-\infty,  ~+ \infty \}$, where $w (0) = + \infty$.  

Let $a \in S$ be nonzero, fix $m$ such that $a \in R_m$, and denote $a_n R_n$ to be the transform of $a 
R_m$ in $R_n$ for all $n \ge m$.
Let $\m_n$ denote the maximal ideal of $R_n$ and $x_n$ be such that $\m_n R_{n+1} = x_n R_{n+1}$.
In Theorem~\ref{w-values}, we prove that 
	$$w(a) ~ =  ~   \sum_{n=m}^{\infty} \ord_{n} (a_n) w (x_n).$$
 
 This allows us to define  the  invariant 
  $\tau = \sum_{n=0}^{\infty} w (x_n)$ 
  associated with the sequence $\{ R_n \}_{n \in \mathbb N}$.
  In Theorem~\ref{arch},  for $S$ with $\dim S \ge 2$, we prove 
  that $\tau < \infty   \iff  S$ is archimedean   $\iff    w$ 
   defines a valuation on $F$ that
  dominates $S  \iff S$ is dominated by a rank 1 valuation domain.
  
By construction,  $w(x) = 1$, so the image of $w$ contains nonzero rational  values. 
  If  the image of  $w$ also has irrational values,   Proposition~\ref{rational rank} 
  gives an explicit finite upper bound for $\tau$.

Let $F^{\times}$ denote the nonzero elements in the quotient field of $S$.
If  $S$ is archimedean  with  $\dim S \ge 2$
we prove in
 Theorem~\ref{v-values}    that the function
 \begin{alignat*}{4}
&v : ~ && F^{\times} &&\rightarrow~&&{\mathbb R} ~ \oplus ~ {\mathbb Z} \\
& ~ &&q &&\mapsto   &&( w (q) , - \edeg (q) ),
\end{alignat*}
defines a  valuation associated to  the boundary valuation ring  $V$ of $S$,
where ${\mathbb R} ~ \oplus ~ {\mathbb Z}$ is ordered lexicographically.
It follows that  $V$  has either rank $1$ or rank $2$.

In Section~\ref{section 7}, we consider the complete integral closure $S^*$ of an archimedean Shannon extension $S$ with $\dim S \ge 2$.
We prove in Theorem~\ref{almost integral} that the almost integral elements in $F$ over $S$ are precisely the elements $a \in T$ such that $w (a) = 0$ and $\edeg (a) > 0$.
Together with Theorem~\ref{w-values}, this allows us to characterize in Theorem~\ref{7.4} whether $S$ is completely integrally closed in terms of the rational dependence of $\tau$ over the subgroup $w (T^{\times})$ of $\mathbb{R}$.
If $S$  is not completely integrally closed,  we prove in Theorem~\ref{7.4} that $S^{*}$ is a generalized Krull domain.
In Examples~\ref{finite construction} and \ref{infinite construction}, we describe 
a method to construct examples of completely integrally closed archimedean Shannon 
extensions that are not valuation domains.

If $S$ is non-archimedean,  we prove in Theorem~\ref{nonarch ew} for $a \in F^\times$ 
that $\edeg (a) > 0$ implies $w(a) = + \infty$  and
the set 
$P_\infty = \{ a \in S \mid w (a) = + \infty \}$ is a prime ideal of both $S$ and $T$.  
Let $xS$ be primary to the maximal ideal of $S$ and denote $P = \bigcap_{n \ge 0} x^nS$.
For a Shannon extension $S$ with $\dim S \ge 2$, we prove in Theorem~\ref{overview} that $S$ is non-archimedean $\iff  T =  ( P :_F P ) \iff  P \ne (0) 
\iff $ every nonmaximal prime ideal of $S$ is contained in $P$ $\iff$ $T$ is a local ring $\iff$ $T$ is the complete integral closure of $S$. 
If $S$ is non-archimedean, we prove in Theorem~\ref{nonarch valuation} that $\edeg$ defines a DVR, $w$ induces a rational rank $1$ valuation on $T / P_{\infty}$, and $V$ is the composite valuation ring of $\edeg$ and $w$.

In general,  our notation is as in Matsumura   \cite{Mat}.  Thus a local ring need not be Noetherian.
 An element $x$ in the maximal ideal $\m$ of a regular local ring $R$ is said
to be a {\it regular parameter} if $x \not\in \m^2$.   It then follows that  the residue class ring $R/xR$ is again a regular local 
ring.  
We refer to an extension ring $B$ of an integral domain $A$ as an {\it overring of} $A$ if $B$ is a subring of the quotient field of $A$.  If, in addition,  $A$ and $B$ are local and the inclusion map $A \hookrightarrow B$ is a 
local homomorphism,  we say that $B$ {\it birationally dominates}  $A$.  We use UFD as an abbreviation for unique factorization domain,  and DVR as an abbreviation for rank 1 discrete valuation ring.  
For the definition of a local quadratic transform, see \cite[Definition 3]{Abh} or \cite{Lip}.

We  thank  Alan Loper and Hans Schoutens  for    correspondence about
infinite directed unions of  local quadratic transformations, and for their collaboration 
in  the article  \cite{HLOST}.

\section{Transform of an ideal}  \label{sec2}

The concept of the transform of an ideal is used extensively in \cite{HLOST}. 
In this article we often   deal with transforms of principal ideals.  
Properties  of the transform of an ideal are  given in Definition~\ref{2.2}  
and Remark~\ref{2.22}.

\begin{definition}   \label{2.2} 
{\rm
Let $A \subseteq B$ be Noetherian UFDs with $B$ an overring of $A$, and 
let    $I$ be  a nonzero ideal of   $A$.   The ideal   $I$ can be written uniquely as $I = P_1^{e_1} \cdots P_n^{e_n}J$, where the $P_i$ are principal prime ideals of $A$, the $e_i$ are positive integers and $J$ is an ideal of $A$  not contained in a principal ideal of $A$ \cite[p.~206]{Lip}.  
  For each $i$, set $Q_i = P_i({A \setminus P_i})^{-1}B  \cap R$.  If $B \subseteq A_{P_i}$, then $A_{P_i} = B_{Q_i}$,  and otherwise $Q_i = B$.
  The {\it transform}\footnote{ The terminology used by Granja in  \cite[p.~1349]{GMR2} 
is strict transform for what Lipman calls the transform.}  
   of $I$ in $B$ is the ideal $$I^B = Q_1^{e_1} \cdots Q_n^{e_n}(JB)(JB)^{-1},$$
  where $(JB)^{-1}$ is the fractional $B$-ideal  consisting   of all elements $x$ in the quotient field of $B$ for which $xJB \subseteq B$.    
Alternatively, $I^B = Q_1^{e_1} \cdots Q_n^{e_n}K$, where $K$ is the unique ideal of $B$ such that 
$JB = xK$  for some  $x \in B$ and  $B$-ideal
$K$  not contained in a proper principal ideal of $B$.}
\end{definition}


  Lipman   \cite[Lemma 1.2 and Proposition 1.5]{Lip}    establishes the following results 
  about transforms: 
  
  \begin{remark}  \label{2.22} {\rm 
   Let   $A \subseteq B \subseteq C$ 
  be Noetherian UFDs   with  $B$ and $C$  overrings of $A$.  Then
   
  \begin{enumerate}
\item   $(I^B)^C = I^C$ for all nonzero  ideals $I$ of $A$.
\item\label{transform products}     $(IJ)^B = I^BJ^B$ for all  nonzero   ideals $I$ and $J$ of $A$.
\item\label{transform primes}    Let  $P$ be a nonzero  principal prime ideal of $A$.  Then the following are equivalent.
\begin{enumerate}
\item[(a)]    $P^B \ne B$. 
\item[(b)]    $B \subseteq A_P$. 
\item[(c)]      $PB \cap A = P$.
\item[(d)]     $P^B  = Q$ is a prime ideal of  $B$ such that $Q \cap A = P$.
\item[(e)]     $P^B  = Q$ is a prime ideal of  $B$   and    $A_P = B_Q$.  
\end{enumerate}
\end{enumerate}}
\end{remark}

Specializing  to the case in which $R$ is a regular local ring, we obtain an explicit formula for the transform of an ideal of $R$ in 
the regular local rings of a  sequence of local quadratic transforms of $R$. 
A proof for Item 3 of Remark~\ref{GR lemma}   is given in \cite[Lemma 3.6 and Remark 3.7]{HKT}.

%
%

\begin{remark} \label{GR lemma} {\rm 
Let $\{(R_i,{\ff m}_i)\}_{i=0}^\infty$ be a sequence of local quadratic transforms of $d$-dimensional regular local rings with $d \ge 2$.
For each $i$, let $x_i$ be an element of ${\m}_i$ such that $\m_iR_{i+1} = x_iR_{i+1}$.  
Let $I$ be an ideal in $R_0$.

\begin{enumerate}
\item $I R_1 = \m_0^{\ord_{R_0}(I)} I^{R_1} = x_0^{\ord_{R_0}(I)} I^{R_1}$ and $I^{R_1} = x_0^{- \ord_{R_0} (I)} I R_1$.\footnote{In \cite[p.~1349]{GMR2}, this last equation is used to define the (strict) transform of a height $1$ prime ideal in $R_1$.}

\item For $n \ge 0$,
$$I R_n = \left( \prod_{i=0}^{n-1} x_i^{\ord_{R_i}(I^{R_i})} \right) I^{R_n} = \left( \prod_{i=0}^{n-1} {\m}_i^{\ord_{R_i}(I^{R_i})} \right) I^{R_n}.$$

\item \label{sequence nonincreasing} The sequence of nonnegative integers $\ord_{R_i} (I^{R_i})$ is nonincreasing.

\end{enumerate}}

   \end{remark}

   
\section{Shannon extensions}  \label{sec3}

In this section we establish the setting in which we work throughout the rest of the paper.
We first recall essential results from \cite{HLOST}.
Let $S$ be a Shannon extension of a regular local ring $R$.
The {boundary valuation ring} $V$ of $S$ is given by
\begin{equation*}
 V = \{ a \in F ~ | ~  \ord_{R_i }(a) \ge 0 \text{ for } i \gg 0 \}.\end{equation*}
The valuation ring $V$ is the unique boundary point for the set of order valuation rings of the $R_i$ with respect to the patch topology on the space of valuation overrings of $R$; see \cite[Section 5]{HLOST}.  
Existence and uniqueness of $V$ is a consequence of the following lemma.   

\begin{lemma}\label{compare} 
$\phantom{}$ \hspace{-.12in} {\rm \cite[Lemma 5.2]{HLOST}}
Let $(R,\m) = (R_0, \m_0)$ be a regular local ring and let $ S = \bigcup_{i \ge 0}R_i$ 
be a Shannon extension of $R$. 
For each nonzero element $a$ in the quotient field of $R$, precisely one of 
following hold: either $\ord_{R_i} (a) > 0$ for $i \gg 0$, $\ord_{R_i} (a) = 0$ for $i \gg 0$, or $\ord_{R_i} (a) < 0$ for $i \gg 0$.
\end{lemma}

In addition to the boundary valuation ring $V$, we work extensively with the {\it Noetherian hull} $T$ of $S$.
The authors establish in \cite[Theorems~4.1 and 5.4]{HLOST} basic properties of $T$ and demonstrate the intersection decomposition $S = V \cap T$.


\begin{theorem} $\phantom{}$ \hspace{-.09in}  {\rm \cite[Theorems 4.1 and~5.4]{HLOST}}    \label{hull} 
   Let $S$ be a Shannon extension of a regular local ring $R$.
Let $N$ be the maximal ideal of $S$,  let $T$ be the intersection of all the 
DVR  overrings of $R$   that properly contain $S$, and let $V$ be the boundary valuation ring of $S$.  
Then:
\label{flat} 
\begin{enumerate}

\item
$S = V \cap T$.  

\item 
There exists $x \in N$ such that $xS$ is $N$-primary, and $T = S [1/x]$ for any such $x$.
It follows that the units of $T$ are precisely the ratios of $N$-primary elements of $S$ and $\dim T = \dim S - 1$.

\item
$T$ is a localization of $R_i$ for $i \gg 0$. In particular, $T$ is a Noetherian regular UFD.

\item
$T$ is the unique minimal proper Noetherian overring of $S$. 

 \end{enumerate} 
\end{theorem}

To simplify hypotheses, we establish a setting for the rest of the article. 

\begin{setting} \label{setting 1}  {\rm  We make the following assumptions throughout the paper. 
\begin{enumerate}
\item
 $R$ is a regular local ring with maximal ideal ${\ff m}$ and quotient field $F$.

\item
$\{ R_n \}_{n \ge 0}$ is an infinite sequence of regular local rings, where $R = R_0$ and $R_{n+1}$ is a local quadratic transform of $R_n$ for each $n \ge 0$.

\item $\dim R = \dim R_n = d \ge 2$ for all $n \ge 0$.
Since Krull dimension does not increase upon taking local quadratic transform, we achieve this condition by replacing $R$ with $R_n$ for some large $n$.

\item $S = \bigcup_{n = 0}^\infty R_n$ denotes the Shannon extension of $R$ along $\{R_n\}$ and $N=\bigcup_{n=0}^\infty {\ff m}_n$ denotes the maximal ideal of $S$.

\item
For each $n \geq 0$, 
denote by $\ord_n:F \rightarrow {\mathbb{Z}} \cup \{\infty\}$ the order valuation of $R_n$ 
and by $V_n = \{q \in F:\ord_n(q) \geq 0\}$ the corresponding DVR.

\item 
Fix $x \in S$ such that $xS$ is $N$-primary and denote by $T = S [1/x]$ the Noetherian hull of $S$; see Theorem~\ref{hull}(2).

\item\label{no proximity}
$T$ is a localization of $R$.
By Theorem~\ref{hull}(2), we achieve this by replacing $R$ with $R_n$ for some large $n$.  

\item 
$V$ denotes the boundary valuation ring for $S$ and $S^*$ denotes the complete integral closure of $S$.

\end{enumerate}}
\end{setting}

Setting~\ref{setting 1}\ref{no proximity} is equivalent to for all $n \ge 0$, $\m_n T = T$.
This, together with Remark~\ref{GR lemma}, implies the following fact:

\begin{remark} \label{transform equals extension}
Assume Setting~\ref{setting 1} and let $I \subset R_n$ be an ideal.
Then for $m \ge n$, Setting~\ref{setting 1}\ref{no proximity} implies that $(I^{R_m}) T = I T = I^{T}$.
\end{remark} 

We separate Shannon extensions into those that are archimedean and those that are non-archimedean, where we use the following definition.

\begin{definition} {\rm An integral domain $A$ is {\it archimedean} if $\bigcap_{n>0} a^n A = 0$ for each nonunit $a \in A$.} 
\end{definition}

\begin{remark}  \label{dim nonarch} {\rm
If $A$ is a  non-archimedean integral  domain,  then $\dim A \ge 2$.  Indeed,
if  there exists a nonzero nonunit $a \in A$ such that $0 \ne b \in \bigcap_{i>0}a^iA$ 
for some $b \in A$, then a maximal ideal containing $a$  cannot  be  a minimal 
prime of  $bA$.   
A  Shannon extension $S$ of $R$ as in Setting~\ref{setting 1}   is archimedean  if and only if $\bigcap_{n>0}x^nS = 0$, where $x$ is as in Setting~\ref{setting 1}(5).   A Shannon extension $S$ with $\dim S = 1$ is a rank 1 valuation 
ring,  cf. \cite[Theorem~8.1]{HLOST}.}
\end{remark}

A Shannon extension $S$ is a directed union of integrally closed domains, and is therefore
integrally closed.  
However, there often exist elements in the field $F$ that are almost integral over $S$ and not in $S$.
If this happens, then $S$ is not completely integrally closed;\footnote{An element $\theta$ in the field of fractions of an integral domain $A$ 
 is  {\it almost integral} over $A$ if the ring $A [\theta]$ is a fractional ideal of $A$.
The integral  domain $A$ is \emph{completely integrally closed} if   each element in the field of 
fractions of $A$ that is almost integral over $A$ is already in $A$.  
The \emph{complete integral closure} of a domain is the ring of almost integral elements in its field of fractions. In general, the complete integral closure of a domain may fail to be completely integrally closed.}
see for example Theorem~\ref{not cic}. 
The complete integral closure $S^*$ of an archimedean Shannon extension $S$ is described by the following theorem.  

\begin{theorem}\label{cic arch} 
$\phantom{}$ \hspace{-.09in}
{\rm \cite[Theorem~6.2]{HLOST}}
Assume notation as in Setting~\ref{setting 1} and assume $S$ is archimedean.
Denote by $W$ the rank one valuation overring of the boundary valuation ring $V$.  
Then:
\begin{enumerate}
\item[{\rm (1)}]
 $S^* ~= ~ N :_F N~ = ~ W \cap T  $.
 \item[{\rm (2)}] 
 $S ~ =  ~S^*   \quad \iff  \quad V ~ = ~ W$.
 \item[{\rm (3)}] 
 If $S \ne S^*$, then $S^*$ is a generalized Krull domain, and 
 $W$ is the unique rank 1  valuation overring   $\mathcal V$ of $S$   such that  
 $ S^* = T \cap \mathcal V$.   
 \end{enumerate}
\end{theorem}

\section{Essential prime divisors of a Shannon extension} \label{essential}
 
 \begin{definition} \label{epd def}
{\rm For an integral  domain $A$, let \begin{center}$\epd(A) ~= ~ \{A_P \mid P$ is a height $1$ prime ideal of $A\}$. \end{center}
The notation is motivated by the fact that if $A$ is a Noetherian 
integrally closed domain, then  $\epd(A)$ is the set of  essential prime divisors of $A$. 
With notation as in Setting~\ref{setting 1},  define 
$$\epd(S/R) ~ =  ~ \big\{\mathcal V  \in \bigcup_{i \geq 0} \epd(R_i) \mid S \subseteq \mathcal V   \big\}.$$ }
\end{definition}

 \begin{discussion}  \label{3.2}  {\rm   The authors show  in \cite[Remark~2.4 and Lemma~3.2]{HLOST} that 
 the set $\epd(S/R)$ consists of the essential prime divisors of $R$ that contain $S$
 along with the order valuation rings of any of the $R_i$ that contain $S$.
Moreover, $S$ is a rank $1$ valuation domain if and only 
if $\epd (S / R) = \emptyset$ \cite[Proposition 4.18]{Sha},\footnote{In this case, the sequence $\{ R_i \}$ is said to   switch  strongly infinitely often.}
and $S$ is a rank $2$ valuation domain if and only if $\epd (S / R)$ consists of a single 
element \cite[Theorem 13]{Gra}.\footnote{In this case, the sequence $\{ R_i \}$ is said to be height 1 directed.}

A Shannon extension $S$ is a rank 1 valuation domain 
if and only if $\dim S = 1$, cf.~\cite[Theorem~8.1]{HLOST}.  
If $S$ is not a rank $1$ valuation domain, then    $\dim S \ge 2$ and 
$\epd(S/R) = \epd(S)$.
In this case, Granja and Sanchez-Giralda \cite[Definition~3]{Gra2} define $\{R_i\}$ to be a  {\it quadratic sequence along a prime ideal $\p$ of $R$}
if the transform of $\p$ in $R_i$ is a proper ideal   of $R_i$  for all $i$,  or equivalently, 
$S \subseteq R_\p$
   \cite[Remark~4]{Gra2}.    

Let $T$ be the Noetherian hull of $S$.
Theorem~\ref{hull}(2) implies that $\epd (S / R) = \epd (T)$.
In addition, we have
\begin{enumerate}
    \item[(1)]     $\{R_i\}$ is a quadratic sequence along $\p$ if and only if $T \subseteq R_\p$,   and 
    \item[(2)]   $\p$ is 
   maximal for the sequence $\{R_i\}$  as in \cite[Definition~6]{Gra2}  if and only if $\p R_\p \cap T$ is a maximal ideal of $T$.  
\end{enumerate}

Assume     $\{R_i\}$ is a quadratic sequence along $\p$ and $R/\p$ is regular.     Let 
$\p_i = \p R_\p \cap R_i$ denote the transform of $\p$ in $R_i$.  
Granja and Sanchez-Giralda \cite[Theorem~8]{Gra2}  prove that  $\p$ is 
   maximal for the sequence $\{R_i\}$  if and only if 
   $S/(\p R_\p \cap S)  =  \bigcup_{i \ge 0}R_i/\p_i$ is a rank 1 valuation domain.}

 \end{discussion}

\begin{definition}
{\rm Assume Setting~\ref{setting 1} and let $p \in R_i$ be a nonzero prime element.
We call $p$ an \emph{essential prime element} of $S/R_i$ if $(R_i)_{p R_i} \in \epd (S / R)$.}
\end{definition}

If $p$ is an essential prime element of $S/R_i$, then it follows from  results 
described in  Theorem~\ref{hull}(2) and  Discussion~\ref{3.2}   
 that  $p T$ is  a  height 1 prime ideal of $T$.  Notice, however, that  $p \notin R_j$ for $j < i$,
 and   $p$ is not  a   prime  element  in $R_j$ for   $j > i$.  The ideal  $p S$ is a proper ideal of $S$,
 but is  not a prime ideal of  $S$.

\begin{proposition} \label{4.4}
Assume Setting~\ref{setting 1}.
\begin{enumerate}
\item
Let $p$ be a prime element of $R_n$. Then $p$ is an essential prime element of $S/R_n$ $\iff$  $(pR_n)^{R_m} \ne R_m$ for all $m>n$ $\iff$ $p \not \in T^\times$.



\item \label{4.4.3}
Let $a \in R_n$. Then $a \in T^\times$ if and only if $(aR_n)^{R_m} = R_m$ for $m \gg n$.

\item \label{6.3} Let $a \in R_n$ be a nonzero nonunit.
Then $a = u \tilde{a}$ in $R_n$, where $u \in R_n \cap T^{\times}$ and $\tilde{a}$ is a possibly empty product of essential prime elements of $S / R_n$.
By convention, an empty product is $1$.

\item \label{4.4.4}  Let $a$ be a nonzero nonunit in $R_n$ and as in~\ref{6.3} write $a = up_1 \cdots p_n$, where $u \in T^\times$ and $p_1,\ldots,p_n$ are essential prime elements of $S/R_n$. For each $i$, let $P_i = p_iR_n$.  Then for each $m \geq n$, we have   $(a R_m)^{R_i} = P_1^{R_i} \cdots P_r^{R_i} R_i$ for $i \gg m$. 

\end{enumerate}
\end{proposition}

\begin{proof}
The first equivalence of item~(1) follows from Remark~\ref{2.22}\ref{transform primes}.
To see equivalence with the third statement, use Remark~\ref{transform equals extension} for the ``$\Leftarrow$'' implication.
The ``$\Rightarrow$'' implication follows from the fact that if $p$ is an essential prime element of $S / R_n$, then $pT$ is a height $1$ prime ideal of $T$.

To see Item~\ref{4.4.3}, let $a \in R_n$.
Since the cases where $a = 0$ or $a$ is a unit are trivial, we may assume $a$ is a nonzero nonunit in $R_n$.
Since $R_n$ is a UFD, we may write $a = p_1 \cdots p_n$, where the $p_i$ are prime elements of $R_n$.
Then from Item~(1) and Remark~\ref{2.22}\ref{transform products}, it follows that $a$ is a unit in $T$ $\iff$ each $p_i$ is a unit in $T$ $\iff$ $(p_i R_n)^{R_m} = R_m$ for $m \gg n$ for all $i$ $\iff$ $(a R_n)^{R_m} = R_m$ for $m \gg n$.

Item~\ref{6.3} follows from (1) and the fact that $R$ is a UFD. 

For item (4), from Remark~\ref{GR lemma} 
we have that $a R_m = u u' P_1^{R_m} \cdots P_r^{R_m}$, where $u'R_m$ is a product of powers of $\m_k R_m, \ldots, \m_{m-1} R_m$.
The assumption in Setting~\ref{setting 1}\ref{no proximity} implies that $\m_k S$ is $N$-primary  for each $k$, so $u' \in T^{\times}$.
Moreover, for $i \ge m$, Remark~\ref{2.22}\ref{transform products}
implies that  $(a R_m)^{R_i} = (uu' R_m)^{R_i} P_1^{R_i} \cdots P_r^{R_i}$.
Since $u, u' \in T^{\times}$, it follows from (2) that  $(u u' R_m)^{R_i} = R_i$ for $i \gg m$.
Hence $(a R_m)^{R_i} = P_1^{R_i} \cdots P_r^{R_i} R_i$ for $i \gg m$.
\end{proof}

%

\section{Asymptotic behavior of the order valuations}  \label{sec4}

Assume notation as in  Setting~\ref{setting 1}, so in particular, fix $x \in S$ such that $xS$ is $N$-primary.
In this section we analyze the limit 
  \begin{equation}\label{eq44}  
	\lim_{n \rightarrow \infty} ~~\frac{\ord_n (a)}{\ord_n (x)}
\end{equation} 
for nonzero elements $a \in F$.
This limit plays a key role in our description of Shannon extensions. 
If  $S$ is an archimedean Shannon extension,  we show in Section~\ref{section 4} that  
 the limit given in Equation~1  defines the rank 1 valuation overring of the boundary valuation ring of $S$. 
 If $S$ is a non-archimedean Shannon extension,  we show in Section~\ref{section 6} that the limit 
 given in Equation~1 induces a rational rank 1 valuation on a certain homomorphic image $S/P$ of $S$.

For an ideal  $I$  of $R$,     
Remark~\ref{GR lemma}(3) implies the sequence of nonnegative integers $\ord_{R_i} (I^{R_i})$ is nonincreasing and thus must converge. 
We use the following definition.

\begin{definition} \label{e def} {\rm
Assume Setting~\ref{setting 1} and let $a \in S$ be nonzero.  
Then $a \in R_n$ for some $n \ge 0$.  
Define $\displaystyle \edeg (a) = \lim_{i \rightarrow \infty} \ord_{n +i} ((a R_n)^{R_{n+i}})$.

For  $\frac{a}{b} \in F$, where $a, b$ are nonzero elements in $ S$, let $n \in \mathbb N$ be 
such that $a, b \in R_n$ and 
 define $\edeg (\frac{a}{b}) = \edeg (a) - \edeg (b)$.}
\end{definition}

For nonzero $a \in S$, $\edeg (a)$ is a finite non-negative integer.
A priori, $\edeg (a)$ depends on the starting point $R_n$ and $\edeg (\frac{a}{b})$ depends on the representation of $\frac{a}{b}$ as an element in $F$.
In Lemma~\ref{e facts}, we prove $\edeg (-)$ is independent of starting point and representation.

\begin{lemma}\label{e facts}
Assume Setting~\ref{setting 1}. Let $\edeg$  be as in Definition~\ref{e def}.  Then:
\begin{enumerate}

\item For nonzero $a \in F$, $\edeg (a)$ is well defined.

\item \label{e facts.2}    For each nonzero $a \in R_n \setminus T^\times$,  there exists a factorization 
$a R_n = u \, p_1 \cdots p_r$ in $R_n$, where $u \in R_n \cap T^{\times}$ and $p_1, \ldots, p_r$ are 
 not necessarily distinct
essential prime elements  of $S/R_n$. 
Then $\edeg (a) = \sum_{i=1}^{r} \edeg (p_i)$.

\item
For nonzero $a, b \in F$, $\edeg (a b) = \edeg (a) + \edeg (b)$.

\item\label{e facts 4}
For nonzero $a \in S$, $\edeg (a) = 0$ if and only if $a \in T^{\times}$. 


\end{enumerate}
\end{lemma}

\begin{proof} 
 By Proposition~\ref{4.4}\ref{4.4.3} and~\ref{4.4.4}, for each $a \in S$,  $\edeg(a)$ is independent of the starting point $R_n$.  
Let $a \in S \setminus T^{\times}$ with $a$ nonzero.
Proposition~\ref{4.4}\ref{6.3} implies the factorization 
$a = u \tilde{a} = u p_1 \cdots p_r$ in $R_n$ as in the statement of (2), so we have $a R_n = u p_1 \cdots p_rR_n$. By Proposition~\ref{4.4}\ref{4.4.4}, $\edeg(a) = \edeg(p_1 \cdots p_r) = \edeg(p_1) + \cdots + \edeg(p_r)$, where the latter assertion follows from the fact that  $\ord_i(a) = \ord_i(p_1) + \cdots + \ord_i(p_r)$ for all $i$. This verifies (2). Item (3), as well as the fact that $\edeg(a)$ is well defined for all $a \in F$, now follow from (2) and the fact that $R_n$ is a UFD.  Item (4) is a consequence of Proposition\ref{4.4}\ref{4.4.3}.  
\end{proof}

\begin{remark}\label{e trivial case}
Assume Setting~\ref{setting 1}.
From Lemma~\ref{e facts}\ref{e facts.2} it follows that $\epd (S / R) = \emptyset$ if and only if $\edeg (a) = 0$ for all nonzero $a \in S$.
Thus as in Discussion~\ref{3.2}, $\edeg (a) = 0$ for all nonzero $a \in S$ if and only if $S$ is a rank $1$ valuation domain.

\end{remark}

\begin{lemma}\label{factorization}
Assume Setting~\ref{setting 1} and let $a, b \in S$ be nonzero.
For $n \gg 0$, in $R_n$, there exist factorizations $a = u \tilde{a}$ and $b = v \tilde{b}$, where $u, v \in R_n \cap T^{\times}$, $\tilde{a}, \tilde{b}$ are products of essential prime elements of $S/R_n$,  $\ord_{n}(\tilde{a}) = \edeg (a)$ and $\ord_n (\tilde{b}) = \edeg (b)$.
Furthermore, for any such factorization,
\begin{enumerate}
\item[{\rm (1)}] If $\ord_{n} (a) \ge \ord_{n} (b)$ for $n \gg 0$, then $v$ divides $u$ in $R_n$ for $n \gg 0$.
\item[{\rm (2)}] If $\ord_{n} (a) = \ord_{n} (b)$ for $n \gg 0$, then $v R_n = u R_n$ for $n \gg 0$, and $\edeg (a) = \edeg (b)$.
\end{enumerate}
\end{lemma}

\begin{proof}
By replacing $R$ with $R_n$ for some sufficiently large $n$, we may assume that $a, b \in R$.
As in Proposition~\ref{4.4}\ref{6.3}, we may write $a = u \tilde{a}$ and $b = v \tilde{b}$, where $\tilde{a}, \tilde{b}$ are the products of essential primes of $S / R$ and $u, v \in R \cap T^{\times}$.
By again replacing $R$ with $R_n$ for sufficiently large $n$, we may assume that $\ord_{0} (\tilde{a}) = \edeg (a)$ and $\ord_{0} (\tilde{b}) = \edeg (b)$ as in Lemma~\ref{e facts}(2).

(1) Assume that $\ord_{n} (a) \ge \ord_{n} (b)$ for $n \gg 0$.
By factoring out their greatest common divisor in $R$, we may assume $a, b$ are relatively prime in $R$.   It suffices to show $v$ is a unit in $R$. 
We proceed as in the proof of \cite[Lemma 5.2]{HLOST}.
Denote $a_0 = a$, $b_0 = b$, $Q_0 = (a_0, b_0) R_0$, and let $Q_i = (a_i, b_i)$ be the transform of $Q_0$ in $R_i$, so $Q_{i} = \m_i^{\ord_{i} Q_i} Q_{i+1}$.
Let $e = \lim_{n \rightarrow \infty} \ord_{n} (Q_n)$.  Since $\ord_n(a) \ge \ord_n(b)$ for $n\gg 0$, we have
 $e = \lim_{n \rightarrow \infty} \ord_{n} (b_n)$.
By replacing $R$ with $R_n$ for some $n \gg 0$, 
we may assume that for all $i \ge 0$, $\ord_{i} (b_i) = e$.
Thus $b_i R_i$ is the transform in $R_i$ of the principal ideal $b_0 R_0$.

Consider the factorization $b = v \tilde{b}$ as above.
Let $v_i R_i$, $\tilde{b}_i R_i$ be the transforms in $R_i$ of $v R_0$, $\tilde{b} R_0$, respectively, so that $b_i R_i = v_i \tilde{b}_i R_i$.
By Remark~\ref{GR lemma}\ref{sequence nonincreasing}, $\ord_{i} (\tilde{b}_i)$ and $\ord_{i} (v_i)$ are both nonincreasing.
Since $v \in T^{\times}$,  we have  $v_i R_i = R_i$ for $i \gg 0$, 
and  $\ord_i (v_i) = 0$ for $i \gg 0$.
Thus $\ord_i (\tilde{b}_i) = e$ for $i \gg 0$.
Since $e \ge \ord_{0} (\tilde{b}) \ge \ord_i (\tilde{b}_i) = e$, we conclude that $\ord_{0}(v) = 0$; that is, $v$ is a unit in $R$, and thus divides $u$ in $R$.

(2) This follows by applying (1) twice.
\end{proof}

\begin{corollary}\label{e value 0}
Assume Setting~\ref{setting 1}.
If $a \in F^\times$ is such that $\edeg (a) = 0$, then there exists $u \in T^\times$ such that $\ord_n (a) = \ord_n (u)$ for all $n \gg 0$.
\end{corollary}

\begin{proof}
Write $a = \frac{b}{c}$ with $b, c \in S$ and apply Lemma~\ref{factorization} to obtain $b = u \tilde{b}$ and $c = v \tilde{c}$.
Then $\frac{u}{v} \in T^{\times}$ and $\ord_n (\frac{u}{v}) = \ord_n (a)$ for all $n \gg 0$.
\end{proof}

We prove in Theorem~\ref{limit exists} that the limit described in Equation~\ref{eq44} exists.

\begin{theorem}\label{limit exists}
Assume  notation as in Setting~\ref{setting 1}.
For  $a \in F^{\times}$   the (possibly infinite) limit
	$$\lim_{n \rightarrow \infty}  ~~\frac{\ord_n (a)}{\ord_n (x)}$$
exists.
\end{theorem}

\begin{proof}

We construct a Dedekind cut as follows:
	$$A ~ =  ~\left\{ \frac{p}{q} \in {\mathbb{Q}} ~ \mid  ~\frac{\ord_n (a)}{\ord_n (x)} ~ \ge ~   \frac{p}{q} ~ \text{ for } n \gg 0 \right\}.$$
	$$B ~=  ~\left\{ \frac{p}{q} \in {\mathbb{Q}}  ~\mid  ~\frac{\ord_n (a)}{\ord_n (x)} ~ <  ~\frac{p}{q} ~ \text{ for } n \gg 0 \right\}.$$
For ${p}/{q} \in {\mathbb{Q}}$ (assuming $q > 0$), $p/q \in A$ is equivalent to $\ord_{n} (a^q) \ge \ord_{n} (x^p)$ for $n \gg 0$, and ${p}/{q} \in B$ is equivalent to $\ord_{n} (a^q) < \ord_{n} (x^p)$ for $n \gg 0$.
By the construction of $A$ and $B$, it follows that for $r \in A$, $r < s$ for all $s \in B$.
By Lemma~\ref{compare}, we conclude that $A \cup B = {\mathbb{Q}}$.
Thus the limit in the statement of the theorem is equal to $\sup A = \inf B$.
\end{proof}

In view of Theorem~\ref{limit exists}, we define a function $w$ as follows:

\begin{notation}\label{w-function}{\rm
 Assume Setting~\ref{setting 1} and 
define  $w : F \rightarrow {\mathbb{R}} ~  \cup   ~  \{-\infty,~+ \infty \}$ by   
defining  $w (0) = +\infty$, and for each $q \in F^{\times}$,  
	$$
	w (q) ~ = ~  \lim_{n \rightarrow \infty} \frac{\ord_n (q)}{\ord_n (x)}.
	$$}
\end{notation}

\begin{remark}  \label{w valuation}
{\rm Assume Setting~\ref{setting 1}.
Since $w$ is the limit of valuations, $w$ behaves like a valuation.
In particular, for elements 
 $a, b \in F$,  we have:
\begin{enumerate}
\item
$w (a + b) \ge \min \{ w (a), w (b) \}$, and $w (a + b) = \min \{ (w (a), w (b) \}$ if $w (a) \ne w (b)$.
\item
$w (a b) = w (a) + w (b)$, except in the case where one value is $+\infty$ and the other is $-\infty$.
\item
If $a \ne 0$, then $w (a) = - w (\frac{1}{a})$.
\item
If $A$ is a subring of $F$ such that $w (a) \ne - \infty$ for each $a \in A$, then $$P = \{ a \in A \mid w (a) = + \infty \}$$ is a prime ideal of $A$.
If in addition there exists a nonzero $a \in A$ with $w (a) \ne 0$, then $w$ 
induces a rank $1$ valuation  $w'$  on the quotient field of $A / P$ such that $w(a) = w'(a')$,
where $a'$ is the image of $a$ in $A/P$.
\end{enumerate}}
\end{remark}
\begin{proof}
Items 1, 2, and 3 follow from the fact that $w$ is a limit of order valuations.
For item~4, see \cite[Remark~2, p.~387 and Prop.~4, p.~388]{B}.
Since $a'$ is the coset $a + P$ and the elements in $P$ have $w$ value $+\infty$, we have 
$w(a) = w'(a')$.  
\end{proof}

We establish the basic properties of $w$ with respect to the Shannon extension $S$.

\begin{theorem}\label{w properties}
Assume notation as in Setting~\ref{setting 1}.
\begin{enumerate}

\item
 If $a \in V$, then $w (a) \ge 0$.

\item\label{w properties 3}
If $a \in F$, then $w (a) > 0$ implies that $a \in \mathfrak{m}_V$.

\item\label{w properties 2}
If $a \in S$, then $w (a) > 0$ if and only if $a \in N (= \m_V \cap S)$.

\item
Let $a \in F^{\times}$ be such that $\edeg (a) = 0$.
Then $w (a)$ is finite, and $a \in V$ if and only if $w (a) \ge 0$.

\item  
For each $n \in \mathbb N$ and element $z \in \m_n$, we have
$$
 zR_{n+1} = \m_nR_{n+1} ~  \iff ~ w(z)~ =  ~ \min\{w(y) ~|~ y \in \m_n \}.
$$ 

\end{enumerate}
To summarize items 1, 2, and 3,
$$\{ a \in F \mid w (a) > 0 \} ~\subseteq~ \mathfrak{m}_V ~\subseteq~ V ~\subseteq~ \{ a \in F \mid w (a) \ge 0 \}$$
	$$N = \{ a \in S \mid w (a) > 0 \}.$$
\end{theorem}

\begin{proof}
For item 1, if $a \in V$, then $\ord_{n} (a) \ge 0$ for $n \gg 0$, so $w (a) \ge 0$.
For item 2, if $a \in F$ is such that $w (a) > 0$, then $\ord_n (a) > 0$ for $n \gg 0$, so $a \in \mathfrak{m}_V$.

To see item 3, let $a \in S$.
The ``only if'' implication follows item 2.
To see the ``if'' implication, assume that $a \in N$.
Since the ideal $x S$ is $N$-primary, there is a positive integer $r$ such that $a^r \in x S$, so $a^r / x \in R_n$ for $n \gg 0$.
Then $w (a^r / x) \ge 0$, so $r w (a) > w (x) = 1$, so $w (a) > \frac{1}{r} > 0$.
This proves item 3.

To see item 4, let $a \in F^\times$ be such that $\edeg (a) = 0$.
By Corollary~\ref{e value 0}, there exists $y \in T^\times$ such that $\ord_n (a) = \ord_n (y)$ for $n \gg 0$, so $w (a) = w (y)$ and $a \in V$ if and only if $y \in V$.
Thus to show item 4, we may assume $a = y \in T^\times$.
Theorem~\ref{hull}(2) implies that $a = {u}/{v}$, where $u, v \in N$ are $N$-primary elements of $S$.

By item 3, $w (u) > 0$.
Since $u S$ is $N$-primary, there exists a positive integer $s$ such that $x^s \in u S$, so by the same argument as in item 3, $w (u) \le s$.
Since $w (u)$ is positive and bounded, it is finite.
Similarly $w (v)$ is finite, so we conclude that $w (a) = w (u) - w (v)$ is finite.

The principal $N$-primary ideals are linearly ordered as a set under inclusion \cite[Corollary 5.5]{HLOST}, so the ideals $u S$ and $v S$ are comparable by inclusion.
From the multiplicativity of $w$ and finiteness of $w (u)$ and $w (v)$, we conclude that $a \in S$ if and only if $w (a) \ge 0$.
Since $a \in T^\times$ and $S = T \cap V$, it follows that $a \in V$ if and only if $w (a) = 0$.
This completes the proof of item 4.

For item 5, let $z \in \m_n$.
Then $z \in x_nR_{n+1} = \m_nR_{n+1}$, so we may write $z = x_n a$ for some $a \in R_{n+1}$, where $w(z) = w(x_n) + w(a)$.
We have $w (a) \ge 0$ by item 1, so $w (z) \ge w (x_n)$, where equality holds if and only if $w (a) = 0$.
Thus $w (x_n) = \min \{ w (y) ~|~ y \in \m_n \}$.
Furthermore, item 3 implies that $a$ is a unit in $R_{n+1}$ if and only if $w(a) = 0$.
We conclude that $z R_{n+1} = \m_n R_{n+1}$ if and only if $w (z) = \min \{ w (y) ~|~ y \in \m_n \}$
\end{proof}

The function $w$ can yield infinite values even restricted to nonzero elements of $S$.
We prove in Theorem~\ref{arch} that if $\dim S \ge 2$, then $S$ is archimedean if and only if $w$ takes only finite values on nonzero elements of $F$, in which case $w$ is a valuation.

In Theorem~\ref{w-values},   we give an alternate interpretation of the restriction of the 
function $w$  to $S$.

\begin{theorem}\label{w-values}
Assume Setting~\ref{setting 1}.
Let $a \in S$ be nonzero.
Then $a \in R_m$ for some $m \ge  0$.
Let $a_i R_i$ be the transform 
of $a R_m$ in $R_i$ for all $i \ge m$.
We have 
	$$w(a) ~ =  ~   \sum_{n=m}^{\infty} \ord_{n} (a_n) w (x_n)$$
where $\m_n R_{n+1} = x_n R_{n+1}$.
\end{theorem}

\begin{proof}
By Theorem~\ref{limit exists}, the possibly infinite limit exists.
From Setting~\ref{setting 1}\ref{no proximity}, we have $x_n \in T^{\times}$ for all $n \ge 0$.

For $n \ge m$, using Remark~\ref{GR lemma}, we have $a R_n = \left( \prod_{i=m}^{n-1} x_i^{\ord_{i} (a_i)} \right) a_n R_n$.
Thus for all $n \ge m$ and for all $j \ge 0$,
$$\ord_j (a) = ~ \left( \sum_{i=m}^{n-1} \ord_i (a_i) \ord_j (x_i) \right) ~ + ~  \ord_j (a_n).$$
Dividing both sides by $\ord_j (x)$,
$$\frac{\ord_j (a)}{\ord_j (x)} = ~ \left( \sum_{i=m}^{n-1} \frac{\ord_i (a_i) \ord_j (x_i)}{\ord_j (x)} \right) ~ + ~  \frac{\ord_j (a_n)}{\ord_j (x)}.$$
Taking the limit as $j \rightarrow \infty$ on both sides and applying Theorem~\ref{limit exists} on the middle terms,
$$\lim_{j \rightarrow \infty} \frac{\ord_j (a)}{\ord_j (x)}  ~ = ~ \left( \sum_{i=m}^{n-1} \ord_i (a_i) w (x_i) \right) ~ + ~  \lim_{j \rightarrow \infty} \frac{\ord_j (a_n)}{\ord_j (x)}.$$
So it follows that,
$$\lim_{j \rightarrow \infty} \frac{\ord_j (a)}{\ord_j (x)}  ~ \ge  ~ \sum_{i=m}^{\infty} \ord_i (a_i) w (x_i).$$
Let $\sigma :=  \sum_{i=m}^{\infty} \ord_i (a_i) w (x_i).$
If $\edeg (a) = 0$, then $\ord_i(a_i) = 0$ for $i \gg 0$, so that  
 the sum is finite and the proof is complete by additivity of $w$ as in Remark~\ref{w valuation}.
If $\sigma = \infty$, the limit is $+\infty$ and there is nothing to show.
Assume that $\sigma < \infty$ and $\edeg (a) > 0$.
Let ${p}/{q}$ be any rational number such that the limit in the left hand side of the above inequality  is greater than ${p}/{q}$. Then for $n \gg 0$, $\ord_{n} (a^q) > \ord_{n} (x^p)$.
By Lemma~\ref{factorization}, since $\edeg (x^p) = 0$, it follows that $x^p$ divides $a^q$ in $R_n$ for $n \gg 0$.
But for $n \gg 0$, $a_n$ is a product of essential prime elements by Proposition~\ref{4.4}\ref{4.4.4}, so that $a_n, x$ have no common factors in $R_n$. Thus since by Remark~\ref{GR lemma}(2),  
$a R_n = \left( \prod_{i=m}^{n-1} x_i^{\ord_{i} (a_i)} \right) a_n R_n$,  $x^p$ divides $(\prod_{i=m}^{n-1} x_i^{\ord_{i} (a_i)})^{q}$, so $p = w (x^p) < q \sigma$. Hence  $\frac{p}{q} < \sigma$, and this completes the proof  of   the theorem.  
\end{proof}

In view of Theorem~\ref{w-values}, we single out an invariant $\tau$ naturally associated with the sequence $\{ R_n \}_{n=0}^{\infty}$.

\begin{notation}\label{tau} {\rm
With $w$ as in Notation~\ref{w-function}, we define 	
	 $\tau = \sum_{n=0}^{\infty} w (x_n)$
where $\m_n R_{n+1} = x_n R_{n+1}$ for $n \ge 0$.}
\end{notation}

The invariant $\tau$ relates the function $w$ to the function $\edeg$, c.f. Remark~\ref{w-values representation}.
In Section~\ref{section 4}, we prove for a Shannon extension $S$ of dimension at least $2$ that $\tau$ is finite if and only if $S$ is archimedean.
In Section~\ref{section 7}, we use $\tau$ to determine whether $S$ is completely integrally closed in the archimedean case.

By construction, $w (x) = 1$, so the image of $w$ contains rational values. 
In the case where $w$ also takes finite irrational values, the following proposition exhibits an explicit upper bound for $\tau$.

\begin{proposition}\label{rational rank}
Assume Setting~\ref{setting 1}.
Let $y_1, \ldots, y_r \in \m$, where $r \ge 2$.
If $w (y_1), \ldots, w (y_r)$ are finite and rationally independent, then
	$$\tau ~ \: \le \: ~  \frac{w (y_1) + \ldots + w (y_r)}{r - 1}.$$
\end{proposition}

\begin{proof}
We argue as in the proof of \cite[Prop.~7.3]{HKT2}.
We inductively prove that for all $n \ge 0$, 
there are elements $y_1^{(n)}, \ldots, y_r^{(n)} \in \m_n$
 such that $w(y_1^{(n)}), \ldots, w(y_r^{(n)})$ are rationally independent and 
 \begin{equation}\label{rre1}
(r - 1) \left( \sum_{i=0}^{n - 1} w (x_i) \right) ~ +  ~\sum_{j=1}^{r} w (y_j^{(n)}) ~ \le ~ \sum_{j=1}^{r} w (y_j).
 \end{equation}
Taking  $y_j^{(0)} = y_j$, the base case $n = 0$ is clear. 
Assume the claim is true for $n$.
Thus  we have elements $y_j^{(n)} \in \m_n$ such that 
Equation~\ref{rre1} holds.

By Remark~\ref{w valuation},     $z \in \m_n$ has  minimal $w$-value if and only if 
$zR_{n+1} = \m_nR_{n+1}$.   Thus  the  set $w (x_n), w (y_1^{(n)}), \ldots, w (y_r^{(n)})$ has 
rational rank at least $r$.  
By re-ordering, without loss of generality   $w (x_n), w (y_2^{(n)}), \ldots, w (y_r^{(n)})$ are rationally independent.
Set $y_1^{(n+1)} = x_n$ and set $y_j^{(n+1)} = \frac{y_j^{(n)}}{x_n}$ for $2 \le j \le n$.
It follows that $w (y_1^{(n+1)}), \ldots, w (y_r^{(n+1)})$ are rationally independent.
Since $w (x_n) < w (y_j^{(n)})$ for $2 \le j \le n$, we have by Theorem~\ref{w properties}\ref{w properties 2} that  $y_j^{(n+1)} \in \m_{n+1}$ for $2 \le j \le n$.
We also have $w (x_n) \le w (y_1^{(n)})$.
Thus,
\begin{align*}
	(r - 1) \left( \sum_{i=0}^{n} w (x_i) \right) + \sum_{j=1}^{r} w (y_j^{(n+1)})
	\: = \: &  (r - 1) \left( \sum_{i=0}^{n - 1} w (x_i) \right) + (r - 1) w (x_n) + \\  & w (x_n) + \sum_{j=2}^{r} ( w (y_j^{(n)}) - w (x_n) )\\
	\: \le \: & (r - 1) \left( \sum_{i=0}^{n - 1} w (x_i) \right) + \sum_{j=1}^{r} w (y_j^{(n)}) \\
	\: \le \: & \sum_{j=1}^{r} w (y_j),
\end{align*}
where the last inequality is a consequence of Equation~\ref{rre1}.
The conclusion follows.
\end{proof}

\section{Archimedean Shannon extensions}  \label{section 4}

Theorem~\ref{arch}  shows 
  in the archimedean case that 
   the mapping $w$ defined in Notation~\ref{w-function} is a rank 1 valuation   whose valuation ring dominates $S$.  

\begin{theorem}   \label{arch}
Assume Setting~\ref{setting 1} and let  $w$ and $\tau$ be as in Notation~\ref{w-function} and 
Notation~\ref{tau}, respectively.  
If $\dim S \ge 2$, then the following are equivalent.
\begin{enumerate}
\item[{\rm (1)}] $S$ is archimedean.
\item[{\rm (2)}]
$w(q)$  is finite for all nonzero $q \in F$.
\item[{\rm (3)}] 
$w(q)$  is finite for some nonzero $q \in S \setminus T^{\times}$.
\item[{\rm (4)}]  $\tau < \infty$.
\item[{\rm (5)}] $w$ defines  a valuation on $F$ that dominates $S$.
\item[{\rm (6)}] $S$ is dominated by a   rank 1  valuation domain.\footnote{ An archimedian Shannon extension $S$ 
of $R$ is   often  birationally dominated by infinitely many rank 1 valuation domains. This is the case in \cite[Example~4.17]{Sha}.}
 \end{enumerate}

 
\end{theorem}

\begin{proof} (1) $\Rightarrow$ (2) 
Since $w$ is multiplicative, we may assume $q \in N$.
By Theorem~\ref{limit exists}, the limit that defines $w(q)$ exists, so $w(q)$ is finite if the sequence $\{{\ord_{n} (q)}/{\ord_{n} (x)}\}_{n=1}^\infty$ is bounded. 
Since $S$ is archimedean, there is some integer $m \ge 0$ such that $q \notin x^m S$.
Thus ${q}/{x^m} \notin S$. Since  $q/x^m \in T$,  and $S = V \cap T$
by Theorem~\ref{hull}(1),  , we conclude that $q/x^m \not \in V$. Thus 
$\ord_{n} (q) < \ord_{n} (x^m)$ for $n \gg 0$, which implies that ${\ord_{n} (q)}/{\ord_n (x)} < m$ for $n \gg 0$.


(2) $\Rightarrow$ (3) This is clear. 

(3) $\Rightarrow$ (4) 
Since $q \in S \setminus T^{\times}$, it follows from Lemma~\ref{e facts} that $\edeg (q) > 0$, where $\edeg$ is as in Definition~\ref{e def}. 
Theorem~\ref{w-values}  implies that $\edeg (q) \sum_{n \ge m} w (x_n) \le w (q) < \infty$ for some integer $m > 0$, so it follows that $\tau < \infty$.

(4) $\Rightarrow$ (5)
Theorem~\ref{w-values}  implies that     $w(q)$  is finite for all $q \in F^{\times}$.  Thus by Remark~\ref{w valuation},  $w$ 
defines a valuation  ring   that dominates $S$.

(5) $\Rightarrow$ (6) Since the valuation in (5) has values in ${\mathbb{R}}$, it has rank 1, so that (6) is clear. 

(6) $\Rightarrow$ (1) Let $U$ be a  rank 1 valuation domain  that dominates $S$, and let $x \in N$. Since $\dim U = 1$, we have $\bigcap_{n>0}x^nS \subseteq \bigcap_{n>0}x^nU = 0$, so (1) follows.
\end{proof}

\begin{definition} \label{6.13}   {\rm
In the case where $S$ is archimedean, we denote by $W$ the rank 1 valuation domain
defined by $w$. Notice that  $W$ dominates $S$.  }
\end{definition}

\begin{remark}\label{w-values representation}
{\rm Assume that $S$ is archimedean and let $a \in F^\times$.
By Theorem~\ref{w-values}, there exists $y \in T^{\times}$ such that $w (a) = w (y) + e (a) \tau$.}
\end{remark}

\begin{proof}
By Theorem~\ref{arch},  we have  $\tau < \infty$.
We may assume $a \in S$.   It follows that  $a \in R_m$ for some $m \ge 0$.  
As in the proof of Theorem~\ref{w-values},  let $a_i R_i$ be the transform of $a R_m$ in $R_i$ for all $i \ge m$, and let $\m_i R_{i+1} = x_i R_{i+1}$ for all $i \ge 0$.
Let $k \ge m$ be such that $\ord_i (a_i) = \edeg (a)$ for all $i \ge k$.
By Theorem~\ref{w-values},
	$$w (a) = \sum_{n \ge m} \ord_n (a_n) w (x_n).$$
Then
\begin{align*}
	w (a) 
		&= \sum_{n \ge 0} \edeg (a) w (x_n) - \sum_{0 \le n < m} \edeg (a) w (x_n) + \sum_{m \le n < k} (\ord_n (a_n) - \edeg (a) ) w (x_n) \\
		&= \edeg (a) \tau ~ - \sum_{0 \le n < m} \edeg (a) w (x_n) + \sum_{m \le n < k} (\ord_n (a_n) - \edeg (a) ) w (x_n)  \\
		&= \edeg (a) \tau + w \left( \prod_{0 \le n < m} x_n^{- \edeg (a) } \prod_{m \le n < k} x_n^{\ord_n (a_n) - \edeg (a)}\right).
\end{align*}
Then $y$ is a finite product of integer powers of the  elements $x_0, \ldots, x_{k-1} \in T^{\times}$, and 
we have  $w (a) = w (y) + e (a) \tau$.
\end{proof}

Using the mappings $w$ and $\edeg$,   Theorem~\ref{v-values}    gives  in 
the archimedean case an explicit description of the valuation associated to  the boundary valuation ring.

\begin{theorem}\label{v-values}
Assume Setting~\ref{setting 1}, let $V$ be the boundary valuation 
ring of $S$, and assume that $S$ is archimedean and $\dim S \ge 2$.
Consider the function 
\begin{alignat*}{4}
&v : ~ && F^{\times} &&\rightarrow~&&{\mathbb R} ~ \oplus ~ {\mathbb Z} \\
& ~ &&q &&\mapsto   &&( w (q) , - \edeg (q) ),
\end{alignat*}
where ${\mathbb R} \oplus {\mathbb Z}$ is ordered lexicographically.
Then $v$ is a valuation of $F$ that defines $V$.
It follows that $V$ has either rank $1$ or rank $2$.
\end{theorem}

\begin{proof} From Lemma~\ref{e facts}(4) and  Theorem~\ref{arch} it follows that  $v(q_1q_2) = v(q_1) + v(q_2)$ for all $q_1,q_2 \in F^\times$. 
To prove that $v$ is a valuation, and that $v$ is the valuation associated to $V$, it suffices to show that if $a \in V$, then $v (a) \ge 0$, and if $a \in {\ff M}_V$, then $v (a) > 0$.


Let $a$ and $b$ be nonzero elements of $S$.  
Suppose ${a}/{b} \in V$. We prove that $v (a) \ge v (b)$.
By definition of $V$, $\ord_{n} (a) \ge \ord_{n} (b)$ for $n \gg 0$.
It follows that $w (a) \ge w (b)$, and if $w (a) > w (b)$, then $v (a) > v (b)$, so we may assume $w (a) = w (b)$.
By Lemma~\ref{factorization}, for a fixed large $n$, there exist in $R_n$ factorizations $a = \alpha \beta \tilde{a}$ and $b = \alpha \tilde{b}$, where $\ord_{n} (\tilde{a}) = \edeg (a)$ and $\ord_{n} (\tilde{b}) = \edeg (b)$.
Let $\tilde{\tau} = \sum_{i=n}^{\infty} w (x_i)$.
Theorem~\ref{w-values} implies that $w (\tilde{a}) = \edeg (a) \tilde{\tau}$ and $w (\tilde{b}) = \edeg (b) \tilde{\tau}$.
Since $w (a) = w (b)$, $w (\tilde{a}) + w (\beta) = w (\tilde{b})$.
Thus $\edeg (a) \tilde{\tau} + w (\beta) = \edeg (b) \tilde{\tau}$.
Since $w (\beta) \ge 0$, it follows that $\edeg (a) \le \edeg (b)$, so we conclude that $v (a) \ge v (b)$.

If in addition ${a}/{b} \in {\ff M}_V$, we prove that $v (a) > v (b)$.
Since $\ord_{n} (a) > \ord_{n} (b)$, $\ord_{n} (\beta) > 0$.
Thus $w (\beta) > 0$, so $\edeg (a) < \edeg (b)$, so $v (a) > v (b)$.
\end{proof}

The following is immediate.

\begin{corollary}\label{rank 1 overring}
Assume Setting~\ref{setting 1}, and assume  $S$ is archimedean with $\dim S \ge 2$.  
Then the valuation domain $W$ of Definition~\ref{6.13} is the rank $1$ valuation overring of the boundary valuation ring $V$,
and one of the following two statements holds. 
\begin{enumerate}
\item[{\rm (1)}]    There exist nonzero $a, b \in S$ such that $w(a) = w(b)$ and  $\edeg(a) \ne  \edeg(b)$. In
this case 
 $V$ has rank $2$.  
\item[{\rm (2)}]   For nonzero  $a, b \in S$  with   $w(a) = w(b)$,  we have $\edeg(a) = \edeg(b)$.  In this case
$V = W$.
\end{enumerate}
\end{corollary}

\section{The complete integral closure of an archimedean Shannon extension}\label{section 7}

Let $S$ be a Shannon extension of $R$ as in Setting~\ref{setting 1}.
If $S$ is non-archimedean, then the complete integral closure of $S$ is $S^{*} = T$; see Theorem~\ref{overview}.
The archimedean case is more subtle, and $S$ may or may not be completely integrally closed.
Theorem~\ref{almost integral} describes the complete integral closure of an archimedean Shannon extension.





\begin{theorem} \label{almost integral}
Assume Setting~\ref{setting 1} and let $w$  be as in Definition~\ref{6.13}.
Assume that $S$ is archimedean with $\dim S \ge 2$.
Let $y, a \in S$ be such that $a$ is nonzero and $yS$ is an $N$-primary ideal.  
Then the following are equivalent.
\begin{enumerate}
\item[{\rm (1)}] $\frac{a}{y} \not \in S$ and $\frac{a}{y}$  is almost integral over $S$.


\item[{\rm (2)}] $w (y) = w (a)$ and $aS$ is not $N$-primary.

\item[{\rm (3)}] $N = (y S :_S a)$.

 \end{enumerate}
Moreover, every element of $S^* \setminus S$ has the form $\frac{a}{y}$ for some $a, y$ with the stated properties.
\end{theorem} 

\begin{proof} 
Let $V$ denote the boundary valuation ring of $S$.

  (1) $\Rightarrow$ (2)
By Theorem~\ref{hull}(1), $S = V \cap T$, and by Theorem~\ref{cic arch},  $S^{*} = W \cap T$. Since $\frac{a}{y} \in T$ and $\frac{a}{y} \in S^{*} \setminus S$, it follows that $\frac{a}{y} \in W \setminus V$.
Since $\frac{a}{y} \in W$, $w (a) \ge w (y)$.
By Theorem~\ref{v-values}, it follows that $w (a) = w (y)$ and $e (\frac{a}{y}) > 0$.
Thus Lemma~\ref{e facts} implies $a S$ is not $N$-primary.

(2) $\Rightarrow$ (1)   
Since $a S$ is not $N$-primary and $y S$ is $N$-primary,  by Lemma~\ref{e facts},  
$e (a) > 0$ and $e (y) = 0$.
It follows from Theorem~\ref{v-values} that $\frac{a}{y} \notin V$, so $\frac{a}{y} \notin S$.
Since $a \in S$ and $\frac{1}{y} \in T$, it follows that $\frac{a}{y} \in T$.
Since $w (\frac{a}{y}) = 0$, $\frac{a}{y} \in W$.
Thus $\frac{a}{y} \in W \cap T = S^{*}$.

(1) $\Rightarrow$ (3) By Theorem~\ref{cic arch}, $(N:_FN)$ is the complete integral closure of $S$,
  and by assumption $\frac{a}{y}$ is almost integral over $S$.  Hence 
  $\frac{a}{y} \in N^{-1}$.    Since $N$ is a maximal ideal of $S$ and $\frac{a}{y} \not \in S$, it 
  follows  that $N = (yS:_Sa)$.  
  
  (3) $\Rightarrow$ (1) By  assumption,  $N = (yS:_Sa)$.   Hence 
  $\frac{a}{y} \in (N:_FN)\setminus S$.   Since $(N:_FN)$ is the complete integral closure of $S$, 
  (1) follows.

Finally, we show every almost integral element $\frac{a}{y} \in S^{*} \setminus S$ with $a, y \in S$ has the property that $y S$ is $N$-primary.
Since $\frac{a}{y} \notin S$,   we have  $y \in N$.
Since $S^{*} = W \cap T$, it follows that $\frac{a}{y} \in T$.
By factoring out the common essential prime factors of $a, y$ in $R_n$ for $n \gg 0$ as in Proposition~\ref{4.4}\ref{6.3}, we may assume that $a, y$ have no common factors in the UFD $T$.
Therefore $y \in T^{\times}$, so $y S$ is $N$-primary.
\end{proof}

  Theorem~\ref{almost integral}  shows 
that the existence of almost integral elements depends on the existence of a pair of elements $a, y \in S$ with $w (y) = w (a)$ and $0 = e (y) < e (a)$.
Using Theorem~\ref{w-values}, we prove in Theorem~\ref{not cic} that $S$ is completely integrally closed if and only if $\tau$ is rationally independent over the 
subgroup $w (T^{\times})$ of $\mathbb R$.

\begin{theorem} \label{not cic}
Assume Setting~\ref{setting 1} and let $w$ and $\tau$ be as in Definition~\ref{6.13} and Notation~\ref{tau}.
Assume that $S$ is archimedean with $\dim S \ge 2$.
Then the following are equivalent.
\begin{enumerate}
\item[{\rm (1)}] $S$ is not completely integrally closed.


\item[{\rm (2)}] $V \subsetneq W$.

\item[{\rm (3)}] $\tau$ is rationally dependent over the subgroup $w(T^\times)$ of ${\mathbb{R}}$.\footnote{A real number $\lambda \in {\mathbb{R}}$ is rationally dependent over a subgroup $G \subseteq {\mathbb{R}}$ if and only if $d \lambda \in G$ for some positive integer $d$.}
\end{enumerate}
\end{theorem}

\begin{proof}
(1) $\Rightarrow$ (2)
If $V = W$, then Theorem~\ref{cic arch} implies $S$ is completely integrally closed.


(2) $\Rightarrow$ (3)
Theorem~\ref{v-values} implies there exists an element $x \in F$ such that $w (x) = 0$ and $e (x) \ne 0$.
Then Remark~\ref{w-values representation} implies $\tau$ is rationally dependent over $w (T^{\times})$.

(3) $\Rightarrow$ (1)
By rational dependence, for some positive integer $d$ there exists $c \in T^{\times}$ such that $w (c) = d \tau$.
Since $\dim S \geq 2$, there exists $a \in S$ such that $\edeg (a) > 0$ by Remark~\ref{e trivial case}.
By Remark~\ref{w-values representation}, there is some $y \in T^{\times}$ such that $w (a) = w (y) + \edeg (a) \tau$.
Then
	$$w (a^d) = d w (y) + d \edeg (a) \tau = w (y^d c^{\edeg (a)}).$$	
By Lemma~\ref{e facts}\ref{e facts 4}, $a \notin T^{\times}$, thus $a$ is not $N$-primary.
By Theorem~\ref{almost integral}, $\frac{a^d}{y^d c^{\edeg (a)}}$ is almost integral over $S$ and not in $S$.
\end{proof}

\begin{remark} 
With notation as in Theorem~\ref{not cic}, Theorem~\ref{v-values} implies that the condition $V \subsetneq W$ is equivalent to the condition that 
$\dim V = 2$.
\end{remark}

 Gilmer in  \cite[page~524]{Gil}  defines  an integral domain  $A$ with quotient field  $K$  to 
be  a \emph{generalized Krull domain}  if 
there is a set $\mathcal{F}$ of rank $1$ valuation overrings of $A$  such that:
(i) $A = \bigcap_{\mathcal V \in \mathcal{F}} \mathcal V$; (ii) for  
each $(\mathcal V, {\ff M}_{\mathcal V}) \in \mathcal{F}$,
 we have  $\mathcal V = A_{{\ff M}_{\mathcal V}  \cap A}$; and (iii) 
$\mathcal{F}$ has finite character; that is, if $x \in K$ is nonzero, then $x$ is a nonunit in
only  finitely many valuation rings of $\mathcal{F}$. 
The class of generalized Krull domains  
 has been studied by a number of authors; see for example \cite{Gri,Gri2,HO,PT,Pir,Rib}.
%

\begin{theorem}\label{7.4}
Assume Setting~\ref{setting 1}  and   let $w$, $W$ and $\tau$ be as in Notation~\ref{w-function}, Definition~\ref{6.13} and Notation~\ref{tau}.
If $S$ is archimedean and not completely integrally closed,
then 
\begin{enumerate}
\item[{\rm (1)}]  $N$ is a height $1$ prime of $S^{*}$, and $S^{*}_{N} = W$.
\item[{\rm (2)}]  Every height $1$ prime ideal of $S^{*}$ is the radical of a principal ideal.
\item[{\rm (3)}]   If $\tau \in w(T^\times)$, then every height 1 prime of $S^*$ other than $N$ is principal.
\item[{\rm (4)}]  $S^{*}$ is a generalized Krull domain.
\end{enumerate}
\end{theorem}

\begin{proof}
By Theorem~\ref{cic arch}, $N$ is the center of $W$ on $S^{*}$, and $S^{*}$ has the representation
$$S^* = W \cap T = W \cap \bigcap_{\substack{\p \in \Spec T \\ \hgt \: \p = 1}} T_{\p}.$$
Since the representation of $T$ as the intersection of its localizations at height $1$ primes has finite character, 
so does this representation of $S^{*}$ as the intersection of valuation domains.  By 
\cite[Lemma 2.3]{HO2},  
the height $1$ prime ideals of $S^{*}$ are a subset of the contractions of the height $1$ primes of $T$,
 along with  possibly $N$.   By \cite[Lemma 1.1]{HO2},  to show $S^{*}_N = W$ and that $N$ is a height $1$ 
 prime of $S^*$, it suffices to show that $\p \cap S^{*} \not\subseteq N$ for each height $1$ prime $\p$ of $T$.


Let $\p = p T$ be a height $1$ prime ideal of $T$.
Then by Lemma~\ref{e facts}(4), $\edeg (p) > 0$.
Theorem~\ref{not cic} implies that $d \tau = w (y)$ for some integer $d > 0$ and some  $N$-primary element $y \in N$.
By Remark~\ref{w-values representation}, $w (p) = \edeg (p) \tau + w (u)$ for some $u \in T^{\times}$.
Denote $q = \frac{p^d}{y^{\edeg (p)} u^d}$.
Thus 
\begin{align*}
	w (q) &= w \left( \frac{p^d}{y^{\edeg (p)} u^{d}} \right) \\
		&= d w (p) - ( \edeg (p) w (y) + d w (u) )\\
		&= d (\tau \edeg (p) + w (u) ) - ( d \tau \edeg (p) + d w (u) ) \\
		&= 0.
\end{align*}
Since $q \in \p$ and $q$ is a unit of $W$, $\p \cap S^{*} \not\subseteq N$.
This completes the proof of (1).

Furthermore, $\p \cap S^{*} = \sqrt{q S^{*}}$.
To see this, let $a \in \p \cap S^{*}$.
Since $a^d \in \p^d = q T$, it follows that $\frac{a^d}{q} \in T$, and since $w (a) \ge 0$ and $w (q) = 0$, it follows that $\frac{a^d}{q} \in W$.
Thus $\frac{a^d}{q} \in S^{*}$, so $a^d \in q S^*$.
Since the only remaining height $1$ prime ideal $N$ of $S^{*}$ is also the radical of a principal ideal, this completes the proof of (2).

If $\tau \in w(T^\times)$,  we may take $q = \frac{p}{y^{\edeg (p)} u}$.  It follows that $\p \cap S^* = qS^*$. 
This proves (3).   Since 
$S^*$ satisfies the conditions of a generalized Krull domain, (4) follows.
\end{proof}



We indicate how to obtain completely integrally closed Shannon extensions that are not valuation domains.
We use the following observations about rank $1$ valuations that birationally dominate a sequence of local quadratic transforms.

\begin{example}\label{finite construction}
 Let $\sigma > 2$ be an irrational real number.
Starting from $R = R_0$, we inductively define a sequence of local quadratic transforms of $2$-dimensional regular local rings $R_i$ with regular system of parameters $\m_i = (x_i, y_i)$,
		$$(R_0, \m_0) \subseteq \ldots \subseteq (R_n, \m_n).$$
We show that it  is possible to choose the sequence $R_i$ so that 
 every rank $1$ valuation ring $\mathcal V$  that  birationally dominates  
 $R_n$  has the following property:  if we choose a valuation $\v$ for $\mathcal V$ such that  
$\v (x_0) = 1$,   then it follows that $\sum_{i=0}^{n-1} \v (\m_i) = \lfloor \sigma \rfloor$ 
and $\frac{\sigma - \lfloor \sigma \rfloor}{\v (x_n)} > 2$, where $\v (x_n) = \frac{1}{2^e}$ for some integer $e \ge 2$.

Set $d = \lfloor \sigma \rfloor - 2$.
Then for $0 \le i < d$, set 
	$$R_{i+1} = R_{i} \left[ \frac{y_i}{x_i} \right]_{(x_i, \frac{y_i}{x_i})}, \quad x_{i+1} = x_i, \quad y_{i+1} = \frac{y_i}{x_i}.$$
Then, set
	$$R_{d + 1} = R_{d} \left[ \frac{y_{d}}{x_{d}} \right]_{(x_{d}, \frac{y_{d}}{x_{d}} - 1)}, 
	\quad x_{d+1} = x_{d}, \quad y_{d + 1} = \frac{y_{d}}{x_{d}} - 1.$$
In this construction, $x_{d+1} = x_0$ and $y_d = \frac{y_0}{x_0^d}$.
Let $\mathcal V$ be a rank $1$ valuation birationally dominating $R_{d+1}$ 
with $\v (x_0) = 1$.
By construction of $R_{d+1}$, $\v (\frac{y_d}{x_d}) = 0$, so it follows that 
$\v (y_d) = \v (x_0) = 1$ and $\v (y_0) = d + 1$.
In particular,
	$$\sum_{i=0}^{d} \v (\m_i) = \sum_{i=0}^{d} 1 = d + 1 = \lfloor \sigma \rfloor - 1.$$

Next, let $e$ be an integer such that $2^e (\sigma - \lfloor \sigma \rfloor) > 2$.
Then for $d + 1 \le i < 2^e + d$, set
	$$R_{i+1} = R_{i} \left[ \frac{x_i}{y_i} \right]_{(y_i, \frac{x_i}{y_i})},
	\quad x_{i+1} = \frac{x_i}{y_i}, \quad y_{i+1} = y_i.$$
Set $f = 2^e + d$ for convenience of notation.
Then set
	$$R_{f + 1} = R_{f} \left[ \frac{y_{f}}{x_{f}} \right]_{(x_{f}, \frac{y_{f}}{x_{f}} - 1)},		\quad x_{f + 1} = x_f, \quad y_{f + 1} = \frac{y_{f}}{x_{f}} - 1.$$
In this construction, $y_f = y_{d+1}$ and $x_f = \frac{x_{d+1}}{y_{d+1}^{2^e - 1}}$.
Let $\mathcal V$ be a rank $1$ valuation domain birationally dominating $R_{f+1}$ 
with $\v (x_0) = 1$.
By construction of $R_{f+1}$, it follows that $\v (\frac{x_f}{y_f}) = 0$, so 
$\v (x_{d+1}) = 2^e \v (y_{d+1})$.
Since $\v (x_{d+1}) = \v (x_0) = 1$, it follows that $\v (y_{d+1}) = \frac{1}{2^e} = \v (x_{f+1})$.
Therefore,
	$$\sum_{i=d+1}^{f} \v (\m_i) = \sum_{i=d+1}^{f} \frac{1}{2^e} = (2^e) \frac{1}{2^e} = 1.$$
Set $n = f + 1$.
We then have
	$$\sum_{i=0}^{n-1} \v (\m_i) = \lfloor \sigma \rfloor.$$
and $\v (x_n) = \frac{1}{2^e}$.




\end{example}

\begin{example}\label{infinite construction} 
Let $\sigma > 2$ be an irrational real number.
We construct an example of an infinite  sequence of 
$2$-dimensional regular local rings $(R_i, \m_i)$ such that 
$R_{i+1}$ is a local quadratic transform of $R_i$ for each $i$,  and 
$\sum_{i=0}^{\infty} \v (\m_i) = \sigma$ for the unique rank $1$ 
valuation ring $\mathcal V$ birationally dominating the union with $\v (\m_0) = 1$.
Since $\dim R_i = 2$, we have $\mathcal V = \bigcup_{n=0}^{\infty} R_i$ by \cite[Lemma 12, p. 337]{Abh}.
In addition, we show that the value group of $\mathcal V$ is the additive group of $\mathbb{Z} [\frac{1}{2}]$.
To construct this sequence, we repeat the construction in Example~\ref{finite construction} an infinite number of times.

We start with $\sigma_0 = \sigma$ and $n_0 = 0$.
Given $\sigma_j$ and $R_{n_j}$, we construct the sequence of $2$-dimensional regular local rings from $R_{n_j}$ to $R_{n_{j+1}}$ for some $n_{j+1} > n_j$ as in Example~\ref{finite construction}.
With this construction, for any valuation ring $\mathcal V$ birationally 
dominating $R_{n_j}$ with $\v (\m_{n_j}) = 1$, it follows that
	$$\sum_{i=n_j}^{n_{j+1} - 1} \v (\m_i) = \lfloor \sigma_j \rfloor.$$
Set $\sigma_{j+1} = 2^{e_j} ( \sigma_j - \lfloor \sigma_j \rfloor)$, 
where $e_j$ is as in Example~\ref{finite construction}.

We repeat this construction to obtain an infinite sequence of local quadratic transforms,
	$$R_0 \subseteq \ldots \subset R_{n_1} \subseteq \ldots \subseteq R_{n_j} \subseteq \ldots$$
Let $S = \bigcup_{n=0}^{\infty} R_n$.
Then $S = \mathcal V$ is a valuation ring.
Fix $\v (\m_0) = 1$.
By Example~\ref{finite construction} and by induction 
$\v (\m_{n_j}) = \prod_{i=0}^{j - 1} \frac{1}{2^{e_i}}$ for every $j$.
Therefore, for any $j$,
	$$\sum_{i=n_j}^{n_{j+1} - 1} \v (\m_i) = \sigma_j \prod_{i=0}^{j-1} \frac{1}{2^{e_i}}. $$
A basic inductive argument yields that
	$$\sigma - \sum_{i=0}^{n_j - 1} \v (\m_i) < \prod_{i=0}^{j-1} \frac{1}{2^{e_i}}.$$
We conclude that $\sigma = \sum_{i=0}^{\infty} \v (\m_i)$.
Furthermore, it follows that for every $i \ge 0$, $\v (x_i), \v (y_i) \in \frac{1}{2^k} \mathbb{Z}$ for some $k \ge 0$ (where $k$ depends on $i$ and increases as $i$ increases).
\end{example}

\begin{corollary}
There exists an archimedean Shannon extension $S$ that is completely integrally closed but not a valuation domain.
\end{corollary}

\begin{proof}
Let $\sigma > 2$ be an irrational real number.
From the construction of Example~\ref{infinite construction}, consider the sequence $(R_i, \m_i)$ of local quadratic transforms of $2$-dimensional regular local rings.
Let $\v$ be the rank $1$ valuation for $\mathcal V = \bigcup_{i=0}^{\infty} R_i$ such that $\v (\m_0) = 1$, so that 
$\sum_{i=0}^{\infty} \v (\m_i) = \sigma$ and the value group of $\v$ is the additive group of $\mathbb{Z} [\frac{1}{2}]$.

Let $z$ be an indeterminate, let $x_n R_{n+1} = \m_n R_{n+1}$ for $n \ge 0$, and denote $z_n = \frac{z}{\prod_{i=0}^{n-1} x_i}$.
Then consider the sequence of local quadratic transforms of $3$-dimensional regular local rings, $R_i' = R_i [z_i]_{(\m_i, z_i)}$.
Since $\edeg (z) = 1$, Remark~\ref{e trivial case} implies  the Shannon extension $S = \bigcup_{i=0}^{\infty} R_i'$ is not a rank $1$ valuation ring.
Let $w$ be as in Notation~\ref{w-function} with $w (\m_0) = 1$.
Notice that for $a \in R_i$, we have $\ord_{R_i} (a) = \ord_{R_i'} (a)$, so the restriction of $w$ to the quotient field of $\mathcal V$ is equal to $\v$.
Thus
	$$\sum_{n=0}^{\infty} w (\m_n) = \sigma, \quad w (T^{\times}) = \bigcup_{n=0}^{\infty} \frac{1}{2^n} \mathbb{Z}.$$
Since $\sigma = \tau < \infty$,  Theorem~\ref{arch}  implies that   $S$ is archimedean.     Since 
$\sigma$ is not rationally dependent over $w (T^{\times})$, Theorem~\ref{not cic} implies that
 $S$ is completely integrally closed.
\end{proof}

\section{Non-archimedean Shannon extensions}  \label{section 6}


In this section we describe the Shannon extensions that are not archimedean.
Like the archimedean case, the functions $\edeg:  F^\times \to \mathbb Z$ and $w : F \rightarrow {\mathbb{R}} ~  \cup   ~  \{-\infty,+\infty \}$ from Definition~\ref{e def} and Notation~\ref{w-function} describe the boundary valuation $v$ in terms of a composite.
In the archimedean case, $w$ defines a valuation ring $W$ 
 on $F$, and if $S \ne S^{*}$, then $\edeg$ defines a valuation on the residue field of $W$.
In the non-archimedean case, the situation reverses:
$\edeg$ defines a valuation    ring $E$ on $F$ and $w$ defines a valuation on the residue field of $E$.


\begin{theorem}\label{nonarch ew}
Assume Setting~\ref{setting 1} and assume  that $S$ is non-archimedean. For $a \in F^\times$,  we have:
\begin{enumerate}
\item[{\rm (1)}]   If    $\edeg (a) > 0$, then $w (a) =  +\infty$.
\item[{\rm (2)}] 
The set 
$P_\infty = \{ a \in S \mid w (a) = + \infty \}$ is a prime ideal of both $S$ and $T$.

\end{enumerate} 
\end{theorem}

\begin{proof}
To prove item~1,  we  proceed as in the proof of Theorem~\ref{w-values}.
Let $a \in F^\times$  and  assume that $\edeg (a) > 0$.  Write $a = \frac{b}{c}$ for some $b, c \in N$.
There exist for a fixed large integer $m$ factorizations $b = u \tilde{b}$ and $c = v \tilde{c}$ in $R_m$ as in Lemma~\ref{factorization}, where $u, v \in T^{\times}$, $\ord_m (\tilde{b}) = \edeg (\tilde{b}) = \edeg (b)$, and $\ord_m (\tilde{c}) = \edeg (\tilde{c}) = \edeg (c)$.
That is, fix $m$ sufficiently large so that the orders of the transforms of the principal ideals $\tilde{b} R_m$ and $\tilde{c} R_m$ are constant, namely $\edeg (b)$ and $\edeg (c)$, respectively.
By Lemma~\ref{e facts}(4), $\edeg(u) = \edeg(v) = 0$, so by Theorem~\ref{limit exists}, $w (u)$ and $w (v)$ exist and are finite. Thus to prove that
 $w (a) = + \infty$ it suffices to show that $w (\frac{\tilde{b}}{\tilde{c}}) = + \infty$.

For $n \ge 0$,  we have  $\m_n R_{n+1} = x_n R_{n+1}$, and  by the assumptions of Setting~\ref{setting 1},  $x_n \in T^{\times}$.
For $n \ge m$,  Remark~\ref{GR lemma}  implies that 
$$
	\tilde{b} = \left( \prod_{i=m}^{n-1} x_i^{\edeg (b)} \right) b_n, \quad
	\tilde{c} = \left( \prod_{i=m}^{n-1} x_i^{\edeg (c)} \right) c_n, \quad
	\frac{\tilde{b}}{\tilde{c}} = \left( \prod_{i=m}^{n-1} x_i^{\edeg (a)} \right) \frac{b_n}{c_n}	
$$
for some $b_n, c_n \in R_n$ with $\ord_n (b_n) = \edeg (b)$ and $\ord_n (c_n) = \edeg (c)$.
Thus for all $n \ge m$ and for all $j \ge 0$,
	$$\ord_j \left( \frac{\tilde{b}}{\tilde{c}} \right) = \ord_j \left( \frac{b_n}{c_n} \right) + \edeg (a) \sum_{i=m}^{n-1} \ord_j (x_i).$$
Dividing both sides by $\ord_j (x)$,
				$$\frac{\ord_j \left( \frac{\tilde{b}}{\tilde{c}} \right)}{\ord_j (x)} = \frac{\ord_j \left(\frac{b_n}{c_n} \right)}{\ord_j (x)} + \edeg (a) \sum_{i=m}^{n-1} \frac{\ord_j (x_i)}{\ord_j (x)}.$$
Taking the limit as $j \rightarrow \infty$ and applying Theorem~\ref{limit exists},
	$$w \left( \frac{\tilde{b}}{\tilde{c}} \right) = w \left( \frac{b_n}{c_n} \right) + \edeg (a) \sum_{i=m}^{n-1} w (x_i).$$
Furthermore, for $j \ge n \ge m$,
	$$\frac{b_n}{c_n} = \left( \prod_{i=n}^{j-1} x_i^{\edeg (a)} \right) \frac{b_j}{c_j}, \quad
	\ord_j \left( \frac{b_n}{c_n} \right) = \edeg (a) + \edeg (a) \sum_{i=n}^{j-1} \ord_j (x_i) > 0,$$
so $\ord_j (\frac{b_n}{c_n}) > 0$.
Therefore $w(\frac{b_n}{c_n}) \ge 0$.
It follows that $w (\frac{\tilde{b}}{\tilde{c}}) \ge \edeg (a) \sum_{i=m}^{n-1} w (x_i)$ 
for all $n \ge m$.
Theorem~\ref{arch} implies that $\sum_{i=m}^{\infty} w (x_i) = +\infty$, 
so it follows that $w (\frac{\tilde{b}}{\tilde{c}}) = +\infty$ and thus $w (a) = +\infty$.
This establishes item~1.  

That $P_\infty$ is  a prime  ideal of  $S$  follows from Remark~\ref{w valuation}.
Since $T = S [1/x]$ and $w (x) = 1$, it follows from Remark~\ref{w valuation} that $T$ has no elements of $w$-value $-\infty$ and that $P'_{\infty} = \{ a \in T ~|~ w (a) = +\infty \}$ is a prime ideal of $T$.
By Theorem~\ref{w properties}\ref{w properties 3}, we have $P'_\infty \subseteq V$.
Since $S = T \cap V$ by Theorem~\ref{hull}, we conclude that $P_\infty = P'_\infty$.
 This establishes  item 2.   
\end{proof}

\begin{remark}  \label{describing P}    With notation  as in Theorem~\ref{nonarch ew},  the 
prime ideal 
$$
P_\infty ~ = ~ \{a \in N ~|~ aS \text{ is not } N\text{-primary} \}.
$$
It follows that $P_\infty$ is the unique prime ideal of $S$ of dimension 1,  and   we have
 $T = S_{P_\infty}  = S[1/x]$,  
where $T$ is the Noetherian hull of $S$.    If $a \in N \setminus P_\infty$ and $b \in P_\infty$,  then
$\frac{b}{a}  \in P_\infty \subseteq S$.  Hence $b \in aS$.  
\end{remark}

 In Theorem~\ref{overview}   we characterize among Shannon extensions of dimension at least $2$ those that are non-archimedean.\footnote{A non-archimedean integral domain necessarily has dimension at least 2 by Remark~\ref{dim nonarch}.}

\begin{theorem}  \label{overview}
Assume  Setting~\ref{setting 1}   and that $\dim S \ge 2$. Let $P = \bigcap_{n \ge 1}x^nS$.  
  Then the following are equivalent. 
 \begin{enumerate}
\item[{\rm (1)}]  $S$ is non-archimedean. 

\item[{\rm (2)}]   $S^* = T = (P:P)$.  

\item[{\rm (3)}]  $P$ is a nonzero prime ideal of $S$.  

\item[{\rm (4)}] 
 Every nonmaximal prime ideal  of $S$ is contained in $P$.

\item[{\rm (5)}] 
  $T$ is a regular local ring. 

\end{enumerate}
\end{theorem}

\begin{proof}
 The equivalence of the first three items is established in  \cite[Theorem~6.9 and Corollary~6.10]{HLOST}.  
 
 (1) $\Rightarrow$ (4)  Since $T = S[1/x]$, we have $P = (S:T)$. Since $S$ is not archimedean,  Theorem~\ref{nonarch ew}(2) gives that $P_\infty$ is an ideal of $T$. Thus  $P_\infty \subseteq (S:T) = P$. By  Remark~\ref{describing P}, every nonmaximal prime ideal of $R$ is contained in $P_\infty$,  so this forces $P_\infty = P$. Hence (4)   holds. 
  
 (4) $\Rightarrow$ (5)  Since $T = S[1/x]$ and $x$ is an $N$-primary element of $S$, (4) implies that $PT$ is the unique maximal ideal of $T$.  Thus by Theorem~\ref{flat}, $T$ is a regular local ring.

 (5) $\Rightarrow$ (1) Suppose $S$ is archimedean.  By assumption $\dim S > 1$. Thus there exists $f \in S$ such that $f$ is a non-unit of $T$, so by Lemma~\ref{e facts}\ref{e facts 4},  
  $\edeg (f) > 0$.
Since $S$ is archimedean, there exists an $N$-primary element $y \in S$ such that $w (y) > w (f)$.
With notation as in Theorem~\ref{v-values}, since $v (y) > v (f)$, it follows that $v (f) = v (f + y)$, so $\edeg (f + y) = \edeg (f)$.
Thus by Lemma~\ref{e facts}\ref{e facts 4}, $f + y$ is a non-unit of $T$.
But $(f + y) - f = y$ is a unit of $T$.
Therefore $T$ is not local.
\end{proof}  

\begin{corollary}
Assume Setting~\ref{setting 1}.
Assume that $S$ is non-archimedean with $\dim S \ge 2$ and denote $P = \bigcap_{n \ge 1} x^n S$.
Then $S / P$ is a rank $1$ valuation domain on the residue field $T / P$ of $T$.
Every valuation domain $\mathcal V$ that dominates $S$ has a prime ideal lying over $P$.
\end{corollary}

\begin{proof}
The principal $N$-primary ideals of $S$ are linearly ordered with respect to inclusion \cite[Corollary 5.5]{HLOST},
each principal $N$-primary ideal contains $P$,
and $S / P$ is a $1$-dimensional local domain by Theorem~\ref{overview}.
Thus $S / P$ is a rank $1$ valuation domain.

Let $\mathcal V$ be a valuation domain dominating $S$ and let $Q = \bigcap_{n \ge 1}x^n \mathcal V$.
Then $Q$ is a prime ideal of $\mathcal V$ and $x \not\in Q$.
Since $x^n S \subseteq x^n \mathcal V$, we have $P \subseteq Q$.  
Also, since $x \not\in Q$, $Q \cap S$ is a nonmaximal prime ideal of $S$. 
By Theorem~\ref{overview}, every nonmaximal prime ideal of $S$ is contained in $P$, 
so we conclude that $P = Q \cap S$.  
\end{proof}
 
  




In Theorem~\ref{nonarch valuation} and Corollary~\ref{v-values2}, we give a complete description of the boundary valuation in the non-archimedean case.

\begin{theorem}\label{nonarch valuation}
Assume Setting~\ref{setting 1} and let $\edeg$ and $w$ be as in Definition~\ref{e def} and Notation~\ref{w-function}.
Assume that $S$ is non-archimedean with $\dim S \ge 2$ and denote $P = \bigcap_{n \ge 1} x^n S$.
Then:
\begin{enumerate}
\item
$\edeg$ is a rank 1 valuation on $F$ whose valuation ring $E$ contains $V$. If in addition $R / (P \cap R)$ is a regular local ring, then $E$ is the order valuation ring of $T$.

\item
$w$ induces a rational rank $1$ valuation $w'$ on the residue field $E/\mathfrak  m_E$  of $E$. The valuation ring  $ W'$  defined by $w'$
extends the valuation ring $S / P$,
and the value group of $W'$ is the same as the value group of $S / P$.

\item
$V$ is the valuation ring defined by the composite valuation of $\edeg$ and $w'$.


\end{enumerate}
The following pullback diagram illustrates $V$:
\begin{equation*}
\xymatrix{
    V \ar@{^{(}->}[d] \ar@{->>}[r] & V/\mathfrak m_E =W'\ar@{^{(}->}[d] \\
    E \ar@{->>}[r] & E / \mathfrak{m}_E
}
\end{equation*}
\end{theorem}

\begin{proof}
We prove that the multiplicative function $\edeg$ defines the rank $1$ valuation overring of $V$.
By Theorem~\ref{nonarch ew}, if $a \in F^\times$   and $\edeg (a) > 0$, then 
$w (a) = +\infty$, so $a \in {\mathfrak m}_V$ by Theorem~\ref{w properties}\ref{w properties 3}.
On the other hand, if $a \in F^\times$ and $\edeg (a) < 0$, then $w (a) = - \infty$, so $a \notin V$.
Therefore if $a \in V$ is nonzero, then $\edeg (a) \ge 0$.

To see that $\edeg$ defines a valuation, it suffices to show that nonzero elements 
of $F$  with positive $\edeg$-value are closed under addition.
Let $a, b \in F$ be two such elements, and assume without loss of generality that $(a, b) V = a V$.
Thus $\frac{b}{a} \in V$, so $\edeg (b) \ge \edeg (a)$, and $1 + \frac{b}{a} \in V$, so $\edeg (1 + \frac{b}{a}) \ge 0$.
Therefore $\edeg (a + b) = \edeg(a(1+\frac{b}{a})) =  \edeg (a) + \edeg (1 + \frac{b}{a}) \ge \edeg (a)$.
It follows that $e$ defines a rank 1 discrete valuation ring $E$, and  the maximal ideal  
$\mathfrak m_E$  of $E$  is a prime ideal of $V$.  
Thus  $V_{\mathfrak m_E} = E$;  cf.  \cite[(11.3)]{Nag}.

Assume in addition that $R / (P \cap R)$ is a regular local ring.
Let $\p_i = P \cap R_i$.
Since $S$ is non-archimedean, Theorem~\ref{arch} implies that $\sum_{i=0}^{\infty} w' (\m_i / \p_i) = \infty$.
Thus \cite[Theorem 10]{Gra2} implies that for every element $f \in R_i$ such that $\ord_{R_i} (f) = \edeg (f)$, we have that $\ord_{(R_i)_{\p_i}} (f) = \ord_{R_i} (f)$.
Since $(R_i)_{\p_i} = T$, we conclude $E$ is the order valuation ring of $T$.
This completes the proof of item~1.

For items 2 and 3, we first observe that by Remark~\ref{w valuation}, $w$ induces a rank 1 valuation $w'$ on  $E/\mathfrak  m_E$.
We prove that $V = \{ a \in F \mid w (a) \ge 0 \}$.
For $a \in F$ such that $w (a) \ne 0$, it follows from the definitions  of $V$ and  $w$ 
that $a \in V$ if and only if $w (a) > 0$.   For $a \in F$ such that $w (a) = 0$,
Theorem~\ref{nonarch ew} implies     
that $\edeg (a) = 0$, so Theorem~\ref{limit exists}(5) implies that $a \in V$.
This proves  item~3.

From Proposition~\ref{rational rank}, it follows that $w'$ has rational rank $1$.
Since $P = \m_E \cap S$ from Theorem~\ref{nonarch ew}, the valuation ring $W'$ extends the valuation ring $S / P$.
Corollary~\ref{e value 0} implies that the range of $w'$ is the same as the range of its restriction to the field $T / P$.
This completes the proof of item 2.
\end{proof}


 Corollary~\ref{v-values2}  describes  a valuation $v$ on the field  $F$ that defines
the boundary valuation $V$  of the Shannon extension    in  Theorem~\ref{nonarch valuation}.

\begin{corollary}\label{v-values2}
Assume  notation as in Theorem~\ref{nonarch valuation}.  
 The   boundary valuation $V$ of $S$ has
rank 2 and rational rank 2.  
Fix an element $z \in E$ such that $\m_E = z E$.
For an element $a \in F^{\times}$, define

\begin{alignat*}{4}
&v : ~ && F^{\times} && ~\rightarrow~ ~&&{\mathbb Z} ~ \oplus ~ ~ {\mathbb Q} \\
& ~ &&a &&~ \mapsto   &&\left( \frac{\edeg (a)}{\edeg(z)}, ~ w \left( \frac{a}{z^{\frac{\edeg (a)}{\edeg (z)}}} \right) \right),
\end{alignat*}
where ${\mathbb Z} ~ \oplus  ~{\mathbb Q}$ is ordered lexicographically. 
The  function  $v$ is a valuation on $F$ that defines the boundary
valuation ring $V$.
\end{corollary}

\begin{proof}   Theorem~\ref{nonarch valuation}(4) states  that  
$V/\mathfrak m_E = W'$ and 
 $V_{\mathfrak m_E} = E$. 
Since $\m_E = z E$ and $\edeg$ is a valuation defining $E$,  the fraction
$ \frac{\edeg (a)}{\edeg(z)}$ is an integer and $aE =  z^{\frac{\edeg (a)}{\edeg(z)}}E$. 
Therefore $a$ and $z^{\frac{\edeg (a)}{\edeg(z)}}$  have the same $\edeg$-value and 
their ratio has $\edeg$-value zero,  and hence has finite $w$-value.
 Theorem~\ref{nonarch valuation}(3) says $w'$
 has  rational rank 1.   Thus    $v(F^\times) \subseteq {\mathbb Z} ~ \oplus  ~{\mathbb Q}$.
 The function $v$ is a homomorphism of the multiplicative group $F^\times$ into the 
 additive group ${\mathbb Z} ~ \oplus  ~{\mathbb Q}$,  and for  $a \in F^\times$, we have $v(a) \ge 0 \iff a \in V$.   It follows that $v$  
 is a valuation on $F$ that defines $V$. 
\end{proof}

Unlike the archimedean case  considered in Theorem~\ref{v-values}, the 
boundary valuation in the non-archimedean case defined in Corollary~\ref{v-values2} depends on assigning a
value in the second component to an element $z \in E$ of minimal positive $\edeg$-value.

Theorem~\ref{nonarch valuation} implies that a non-archimedean Shannon extension may be described as a pullback.
In the paper \cite{HOT} in preparation, we are interested in characterizing non-archimedean Shannon extensions as pullbacks.







\begin{thebibliography}{34}


\bibitem{Abh} 
S.~Abhyankar, 
On the valuations centered in a local domain. 
Amer. J. Math. 78 (1956), 321--348. 





\bibitem{B} N.  Bourbaki,
{\it  Commutative   Algebra}      Chapters 1-7,
Springer-Verlag, Berlin,   1989.
















\bibitem{Gil} R.~Gilmer, {Multiplicative ideal theory,} Marcel Dekker, New York, 1972. 



\bibitem{Gra}  A.~Granja, Valuations determined by quadratic transforms of a regular ring. 
J. Algebra 280 (2004), no. 2, 699--718. 


\bibitem{GMR2} A.~Granja, M.~C.~Martinez and C.~Rodriguez,
Valuations with preassigned proximity relations, J. Pure Appl. Algebra 212 (2008), 1347--1366. 


\bibitem{Gra2} A.~Granja and T.~S\'anchez-Giralda.
Valuations, equimultiplicity and normal flatness. 
J. Pure Appl. Algebra 213 (2009), no. 9, 1890--1900.

\bibitem{Gri} M.~Griffin, 
Rings of Krull type,
J. reine angew. Math. 229 (1968), 1--27.


\bibitem{Gri2} M.~Griffin, Families of finite character and essential valuations, Trans. Amer. Math. Soc. 130
(1968), 75--85.





\bibitem{HKT} W.~Heinzer, M.-K.~Kim and M.~Toeniskoetter, 
Finitely supported $*$-simple complete ideals in a regular local ring,
J. Algebra 401 (2014), 76--106.

\bibitem{HKT2} W.~Heinzer, M.-K.~Kim and M.~Toeniskoetter, Directed unions of local quadratic transforms of a regular local ring, preprint. 

\bibitem{HLOST} W.~Heinzer, K.~.A.~Loper, B.~Olberding, H.~Schoutens and M.~Toeniskoetter,  	
Ideal Theory of Infinite Directed Unions of Local Quadratic Transforms, arXiv:1505.06445.  

\bibitem{HO} W.~Heinzer and J.~Ohm,
Defining families for integral domains of real finite character,
Canad. J. Math.  {\bf 24}  (1972), 1170--1177.

\bibitem{HO2} 
W.~Heinzer and J.~Ohm, Noetherian intersections of integral domains. Trans. Amer. Math, Soc. 167 (1972), pp. 291-308.

\bibitem{HOT} W.~Heinzer, B.~Olberding, and M.~Toeniskoetter, Directed unions of regular local rings as pullbacks, in preparation.














\bibitem{Lip} J.~Lipman, On complete ideals in regular local rings. Algebraic geometry and commutative algebra, Vol. I, 203--231, Kinokuniya, Tokyo, 1988.




\bibitem{Mat} H.~Matsumura {Commutative Ring Theory}  Cambridge Univ. Press,  
Cambridge, 1986.




\bibitem{Nag} M.~Nagata,  {Local Rings}, John Wiley, New York, 1962.








\bibitem{PT} E.~Paran and M.~Temkin, 
Power series over generalized Krull domains. 
J. Algebra 323 (2010), no. 2, 546--550. 

\bibitem{Pir} E. Pirtle,  Families of valuations and semigroups of fractionary
ideal classes, Trans. Amer. Math. Soc. 144 (1969), 427--439.


\bibitem{Rib} P.~Ribenboim, Le th\'eor\`eme d'approximation pour les valuations de Krull,
Math. Zeit. 68 (1957/58), 1--18.





\bibitem{Sha}
D.~Shannon, Monoidal transforms of regular local rings. Amer. J. Math. 95 (1973), 294--320. 




\end{thebibliography}
\end{document}